\documentclass[10pt]{amsart}

\usepackage{amsmath}
\usepackage{amssymb}
\usepackage{verbatim}
\usepackage{enumerate}
\usepackage{amscd}
\usepackage[all]{xy}
\DeclareMathAlphabet{\mathpzc}{OT1}{pzc}{m}{it}

\usepackage{setspace} 
\usepackage{indentfirst} 
\usepackage{amsthm} 
\usepackage{graphicx} 
\usepackage{amssymb}
\usepackage{mathrsfs}

\numberwithin{equation}{section}

\newcounter{TmpEnumi}

\newcommand{\Number}{section}
\theoremstyle{plain}
\newtheorem{thm}{Theorem}[\Number]
\newtheorem{lem}[thm]{Lemma}

\newtheorem{prp}[thm]{Proposition}
\newtheorem{cor}[thm]{Corollary}

\newtheorem{cnj}[thm]{Conjecture}

\newtheorem{exa}[thm]{Example}

\newtheorem{dfn}[thm]{Definition}

\newcommand{\beq}{\begin{equation}}
\newcommand{\eeq}{\end{equation}}
\newcommand{\beqr}{\begin{eqnarray*}}
\newcommand{\eeqr}{\end{eqnarray*}}
\newcommand{\bal}{\begin{align*}}
\newcommand{\eal}{\end{align*}}
\newcommand{\bei}{\begin{itemize}}
\newcommand{\eei}{\end{itemize}}
\newcommand{\limi}[1]{\lim_{{#1} \to \infty}}
\newcommand{\alignInd}{\hspace{3em} }

\newcommand{\af}{\alpha}
\newcommand{\bt}{\beta}

\newcommand{\dt}{\delta}
\newcommand{\ep}{\varepsilon}

\newcommand{\et}{\eta}

\newcommand{\io}{\iota}

\newcommand{\ld}{\lambda}

\newcommand{\rh}{\rho}

\newcommand{\Z}{{\mathbb{Z}}}
\newcommand{\R}{{\mathbb{R}}}
\newcommand{\C}{{\mathbb{C}}}
\newcommand{\N}{{\mathbb{Z}}_{> 0}}
\newcommand{\Nz}{{\mathbb{Z}}_{\geq 0}}
\newcommand{\subeq}{\preccurlyeq}

\newcommand{\tsr}{{\operatorname{tsr}}}
\newcommand{\RR}{{\operatorname{RR}}}

\newcommand{\dist}{{\operatorname{dist}}}
\newcommand{\sa}{{\operatorname{sa}}}
\newcommand{\spec}{{\operatorname{sp}}}

\newcommand{\card}{{\operatorname{card}}}

\newcommand{\andeqn}{\,\,\,\,\,\, {\mbox{and}} \,\,\,\,\,\,}

\newcommand{\ca}{C*-algebra}
\newcommand{\uca}{unital C*-algebra}
\newcommand{\ct}{continuous}
\newcommand{\pj}{projection}

\newcommand{\fd}{finite dimensional}
\newcommand{\wolog}{without loss of generality}
\newcommand{\Wolog}{Without loss of generality}

\newcommand{\ifo}{if and only if}

\newcommand{\hme}{homeomorphism}
\newcommand{\mh}{minimal homeomorphism}

\newcommand{\cms}{compact metric space}
\newcommand{\cfn}{continuous function}

\newcommand{\rokhLg}{centrally large}

\newcommand{\Rql}{centrally large}

\newcommand{\Cu}{{\operatorname{Cu}}}

\newcommand{\ov}{\overline}
\newcommand{\SM}{\setminus}
\newcommand{\I}{\infty}

\title[Centrally Large Subalgebras and Stable Rank One]{Permanence
 of stable rank one for centrally large subalgebras and
 crossed products by minimal homeomorphisms}

\author{Dawn Archey}
\author{N.~Christopher Phillips}

\date{30~March 2016}   

\address{Department of Mathematics and Software Engineering,
University of Detroit Mercy,
4001 West McNichols Road,
Detroit MI 48221-3038,
USA}
\email[]{archeyde@udmercy.edu}

\address{Department of Mathematics, University of Oregon,
       Eugene OR 97403-1222, USA}


\thanks{This material is partially based upon work
  of the second author supported by the
  US National Science Foundation under Grants
  DMS-0701076 and DMS-1101742.
  It was also partially supported by the Centre de Recerca
  Matem\`{a}tica (Barcelona) through research visits
  by both authors conducted
  during 2011.}

\begin{document}

\begin{abstract}
We define centrally large subalgebras
of simple unital C*-algebras,
strengthening the definition of large subalgebras
in previous work.
We prove that
if $A$ is any infinite dimensional simple
separable unital C*-algebra which
contains a centrally large subalgebra
with stable rank one,
then $A$ has stable rank one.
We also prove that large subalgebras of crossed product
type are automatically centrally large.
We use these results to prove that if $X$
is a compact metric space which has a surjective
continuous map to the Cantor set,
and $h \colon X \to X$ is a minimal homeomorphism,
then $C^* ({\mathbb{Z}}, X, h)$ has stable rank one,
regardless of the dimension of~$X$
or the mean dimension of~$h$.
In particular,
the Giol-Kerr examples give crossed products with
stable rank one
but which are not stable under tensoring with the Jiang-Su
algebra and are therefore not classifiable in terms of the
Elliott invariant.
\end{abstract}

\maketitle


\setlength{\parindent}{.5in}

\pagenumbering{arabic}
\setcounter{page}{1}

\section{Introduction}\label{Sec:Intro}

\indent
Large and stably large subalgebras
were introduced in~\cite{Ph40}
as an abstraction of Putnam's orbit breaking subalgebra
of the crossed product $C^* (\Z, X, h)$
of the Cantor set by a \mh.
In this paper,
we define a stronger concept,
that of a centrally large subalgebra.
Large orbit breaking subalgebras are also centrally large,
and centrally large subalgebras provide more information about
the containing algebra.
We prove
(Theorem~\ref{tsr1GoesUpExtraLarge})
that if $A$ is a simple unital \ca{}
and $B \subset A$ is a large subalgebra which has stable rank one,
then $A$ also has stable rank one.
If in addition $B$ has real rank zero,
we prove (Theorem~\ref{T_4Y11_RRZ})
that $A$ also has real rank zero.

We use our result to prove the following result
(Theorem~\ref{T_5413_TsrCP})
for crossed products by \mh{s}.
Let $X$ be a compact metric space
which has a surjective continuous map to the Cantor set,
and let $h \colon X \to X$ be a minimal homeomorphism.
Then $C^* ({\mathbb{Z}}, X, h)$ has stable rank one.
This result holds regardless of the dimension of~$X$
or the mean dimension of~$h$.
It is expected that
$C^* ({\mathbb{Z}}, X, h)$ is classifiable in terms of the
Elliott invariant
\ifo{} $h$ has mean dimension zero,
and one direction is the main result of~\cite{EN2}.
Thus our result shows that
crossed products by \mh{s}
can have stable rank one
even when they are not expected to be classifiable.
The examples of~\cite{GK}
confirm that this happens:
they do not have mean dimension zero,
their crossed products
are not stable under tensoring with the Jiang-Su algebra,
but our theorem shows that the crossed products
do have stable rank one.

In~\cite{PhZd},
we will use our result to prove a generalization
of Theorem~\ref{T_5413_TsrCP}.
Let $X$ be a compact metric space with a free minimal action of~$\Z^d$.
Assume $(X, \Z^d)$ has a factor system which
is a free minimal action of~$\Z^d$ on the Cantor set.
Then $C^* ({\mathbb{Z}}, X)$ has stable rank one.
Again, no hypothesis on mean dimension is needed.
We do not include the proof here because the proof of the existence
of the necessary centrally large subalgebra is somewhat involved.
We point out, however,
that the main reason for the abstraction of Putnam's construction
is to be able to handle situations of this sort,
in which there appears to be no explicit construction of the
necessary subalgebra,
only an existence proof.

We give the motivation and background in more detail.
Let $X$ be an infinite compact metric space,
let $h \colon X \to X$ be a homeomorphism,
and let $Y \subset X$ be closed.
Let $u \in C^* ({\mathbb{Z}}, X, h)$ be the standard unitary.
The $Y$-orbit breaking subalgebra
is the subalgebra $C^* ({\mathbb{Z}}, X, h)_Y$
of $C^* ({\mathbb{Z}}, X, h)$
generated by $C (X)$
and all elements $f u$
for $f \in C (X)$ such that $f |_Y = 0$.
(Putnam actually used $u f$ rather than $f u$,
but our choice makes the
relationship with Rokhlin towers less awkward.)
If $h$ is minimal,
$X$ is infinite,
and $Y$ meets each orbit of~$h$ at most once,
then $C^* ({\mathbb{Z}}, X, h)_Y$
is large in $C^* ({\mathbb{Z}}, X, h)$
(Theorem~7.10 of~\cite{Ph40}).

Subalgebras with properties similar to this one
have been used in many other places,
such as the study of crossed products of the Cantor set
by free minimal actions of $\Z^d$
(\cite{Ph10}, but without the name),
the recent proof by Elliott and Niu that if $h \colon X \to X$
is minimal and has mean dimension zero,
then $C^* ({\mathbb{Z}}, X, h)$ is $Z$-stable~\cite{EN2},
and the proof that ``breaking'' the generating unitaries
in an irrational rotation algebra gives an AF~algebra~\cite{EN1}.
See the introduction to~\cite{Ph40}
for a much longer list of applications.

For an infinite dimensional simple unital \ca~$A$
and a stably large subalgebra $B \subset A$,
the following results were proved in~\cite{Ph40}:
\begin{itemize}
\item
$B$ is simple (Proposition~5.2 of~\cite{Ph40})
and infinite dimensional (Proposition~5.5 of~\cite{Ph40}).
\item
If $B$ is stably finite then so is~$A$
(Corollary~6.16 of~\cite{Ph40}),
and if $B$ is purely infinite then so is~$A$
(Proposition~6.17 of~\cite{Ph40}).
\item
The restriction maps
${\operatorname{T}} (A) \to {\operatorname{T}} (B)$
and ${\operatorname{QT}} (A) \to {\operatorname{QT}} (B)$
(on tracial states and quasitraces) are bijective
(Theorem~6.2 of~\cite{Ph40}
and Proposition~6.9 of~\cite{Ph40}).
\item
When $A$ is stably finite,
the inclusion of $B$ in $A$ induces an isomorphism
on the semigroups that remain after deleting
from ${\operatorname{Cu}} (B)$ and ${\operatorname{Cu}} (A)$
all the classes of nonzero projections
(Theorem~6.8 of~\cite{Ph40}).
\item
$B$ and $A$ have the same radius of comparison
(Theorem~6.14 of~\cite{Ph40}).
\end{itemize}
However,
for several further results,
such as ones involving stable rank one and $Z$-stability,
one seems to need a stronger condition:
$B$ must be what we call centrally large in~$A$.

In this paper,
we define centrally large subalgebras.
We prove that a large subalgebra of crossed product type
(Definition~\ref{D-2717CPType} below;
Definition~4.9 of~\cite{Ph40})
is in fact centrally large;
this result covers most of the examples which have appeared
in applications.
We then prove that if $B$ is centrally large in~$A$
and $B$ has stable rank one,
then so does~$A$.
We will consider $Z$-stability elsewhere~\cite{ArPh2}.

For the application to the stable rank of crossed products,
let $h \colon X \to X$ be as above,
and assume that $X$ has a surjective continuous map to the Cantor set.
(The Giol-Kerr examples~\cite{GK}
have this property,
because the have the universal odometer as a factor.)
Then $Y$ as above can be chosen
to be a decreasing intersection of compact open subsets of~$X$.
In this case,
it is known that $C^* ({\mathbb{Z}}, X, h)_Y$
is the direct limit of an AH~system
which has diagonal maps in the sense of~\cite{EHT}.
Moreover (by~\cite{EHT}),
a simple direct limit
of an AH~system with diagonal maps has stable rank one,
even with no assumptions on the dimension growth in the system.

We conjecture,
jointly with Zhuang Niu,
that $C^* ({\mathbb{Z}}, X, h)$ has stable rank one
for a completely arbitrary \mh~$h$
of an infinite compact metric space~$X$.

The original main application was to prove that
of $X$ is a finite dimensional
compact metric space,
and if $\Z^d$ acts freely and minimally on~$X$,
then $C^* (\Z^d, X)$ has stable rank one,
using the large subalgebra of crossed product type
constructed in~\cite{Ph41}.
This particular application has been superseded
before its publication in two different ways.
First, using centrally large subalgebras,
one can in fact prove $Z$-stability,
which by Theorem~6.7 of~\cite{R4} implies stable rank one.
Second,
$Z$-stability (in fact, finite nuclear dimension)
has also been obtained by different
methods in~\cite{Szb}.

The paper is organized as follows.
In Section~\ref{Sec:CuntzComp} we establish
notation and prove some preliminary lemmas
on approximation and Cuntz comparison.
In Section~\ref{Sec:LgDef} we give
the definition of a centrally large
subalgebra and some variations on the theme.
In Section~\ref{Sec:Rq} we prove
that large subalgebras of crossed product type
are centrally large.
In Section~\ref{Sec:MatDec} we prove
some technical lemmas.
In Section~\ref{Sec:RqTsr1} we prove the main theorem
and the result about real rank zero.
Section~7 contains the result
on the stable rank of $C^* ({\mathbb{Z}}, X, h)$
and its application to examples,
including those of Giol and Kerr.

We use the following notational conventions.
If $A$ is a \ca,
then $A_{+}$ denotes the set of positive elements of~$A$.
Also, $A^{+}$ denotes the unitization of~$A$.
(We add a new identity even if $A$ is already unital.)
The notations $M_{\infty} (A)$ and $a \oplus b$
for $a, b \in M_{\infty} (A)$
and $a, b \in K \otimes A$
are described in Section~\ref{Sec:CuntzComp}.
By convention,
if we say that $B$ is a unital subalgebra of a \ca~$A$,
we mean that $B$ contains the identity of~$A$.
Subalgebras of \ca{s}
are always supposed to be C*-subalgebras.
We take $[a, b] = a b - b a$.

We are grateful to Francesc Perera for pointing out
the reference~\cite{RW},
to Zhuang Niu for pointing out the reference~\cite{EHT},
and to George Elliott for pointing
out Corollary~5 of~\cite{CEI}.

\section{Cuntz comparison and approximation lemmas}\label{Sec:CuntzComp}

\indent
We give definitions and notation related to Cuntz comparison,
and one result needed later for which we don't know a reference.
We refer to the list in Lemma~1.4 of~\cite{Ph40},
and the results elsewhere in Section~1 of~\cite{Ph40},
for many more facts.
Also see the survey article~\cite{APT}.
At the end, we prove several approximation lemmas
which will be needed later.

For a \ca~$A$,
let $M_{\infty} (A)$ denote the algebraic direct limit of the
system $(M_n (A))_{n = 1}^{\infty}$
using the usual embeddings $M_n (A) \to M_{n+1} (A)$,
given by
\[
a \mapsto \left( \begin{array}{cc} a & 0 \\ 0 & 0 \end{array} \right).
\]
If $a \in M_m (A)$ and $b \in M_n (A)$,
we write $a \oplus b$ for the diagonal direct sum
\[
a \oplus b
 = \left( \begin{array}{cc} a & 0 \\ 0 & b \end{array} \right).
\]
By abuse of notation,
we will also write $a \oplus b$ when $a, b \in M_{\infty} (A)$
and we do not care about the precise choice of $m$ and $n$
with $a \in M_m (A)$ and $b \in M_n (A)$.
We further choose some isomorphism $M_2 (K) \to K$,
and for $a, b \in K \otimes A$ we use the resulting isomorphism
$M_2 (K \otimes A) \to K \otimes A$ to interpret
$a \oplus b$ as an element of $K \otimes A$.
Up to unitary equivalence which is trivial on~$A$,
the result does not depend on the choice of
the isomorphism $M_2 (K) \to K$.

The following definitions are originally from~\cite{Cz1}.

\begin{dfn}\label{D:CzSGp}
Let $A$ be a \ca, and let $a, b \in (K \otimes A)_{+}$.
\begin{enumerate}
\item\label{D:CzSGp:1}
We say that $a$ is {\emph{Cuntz subequivalent to~$b$ over~$A$}},
written $a \subeq_A b$,
if there is a sequence $(v_n)_{n = 1}^{\infty}$ in $K \otimes A$
such that
\[
\limi{n} v_n b v_n^* = a.
\]
\item\label{D:CzSGp:2}
We say that $a$ and $b$ are {\emph{Cuntz equivalent in~$A$}},
written $a \sim_A b$,
if $a \subeq_A b$ and $b \subeq_A a$.
\end{enumerate}
The relation $\sim_A$ is an equivalence relation.
When there is no possible confusion about the algebra~$A$,
we suppress it in the notation.
\end{dfn}

We will regularly use the following two families of functions.

\begin{dfn}\label{D:MinusEp}
Let $A$ be a \ca,
let $a \in A_{+}$,
and let $\ep > 0$.
We define $(a - \ep)_{+}$ as follows.
Let $f \colon [0, \I) \to [0, \I)$ be the function
\[
f (\lambda)
 = (\lambda - \ep)_{+}
 = \begin{cases}
     0               & \hspace{3em}  0 \leq \lambda \leq \ep
\\
     \lambda - \ep   & \hspace{3em}  \ep < \lambda.
    \end{cases}
\]
Then define $(a - \ep)_{+} = f (a)$
(using \ct{} functional calculus).
\end{dfn}

\begin{dfn}\label{fSubEp}
For $\ep > 0$,
define a \cfn{} $f_{\ep} \colon [0, \I) \to [0, 1]$
by
\[
f_{\ep} (\ld)
 = \begin{cases}
   0                   & \hspace{3em} 0 \leq \ld \leq \ep
\\
   \ep^{- 1} (\ld - \ep) & \hspace{3em} \ep \leq \ld \leq 2 \ep
\\
   1                   & \hspace{3em} 2 \ep \leq \ld.
\end{cases}
\]
\end{dfn}

\begin{lem}\label{Under2Ways}
Let $A$ be a simple unital \ca,
and let $a, b \in A_{+}$ with $\| a \| = \| b \| = 1$.
Then there exists
$c \in A_{+} \setminus \{ 0 \}$ with $\| c \| \leq 1$
such that $c \leq a$ and $c \subeq b$.
\end{lem}

\begin{proof}
Since $A$ is simple and $a, b \in A$ are nonzero,
by Proposition 1.8 of~\cite{Cz} there is a nonzero
$y \in A$ such that $yy^{*} \in {\overline{aAa}}$ and
$y^{*} y \in {\overline{bAb}}$.
Without loss of generality we may assume that $\| y \| \leq 1$,
and so
$yy^{*} \leq 1$.
Set $z = y^{*} a^{1/2}$
and $c = z^* z$.
Then $c = a^{1/2} y y^{*} a^{1/2} \leq a$.
Also,
the discussion after Definition~2.3 of~\cite{KR}
gives the first step in the calculation
$c \sim z z^* = y^* a y \leq y y^* \in {\overline{b A b}}$,
so $c \subeq b$ by Proposition 2.7(i) of~\cite{KR}.
\end{proof} 

We finish this section with several approximation lemmas.
The first is actually a special case of Lemma 2.5.11(1)
of~\cite{LnBook},
but the proof there is a bit sloppy.

\begin{lem}\label{FuncCalcCom}
Suppose $f \colon [0, 1] \to \C$ is continuous.
Then for every $\ep > 0$
there exists $\delta > 0$
such that whenever $D$ is a \ca{}
and $x, z \in D$
satisfy
\[
0 \leq x \leq 1,
\,\,\,\,\,\,
\| z \| \leq 1,
\andeqn
\| [x, z] \| < \delta,
\]
then $\| [f (x), z ] \| < \ep$.
\end{lem}

\begin{proof} 
We first observe that if $\| x \| \leq 1$ and $\| z \| \leq 1$,
then for $n \in \N$ we have
\begin{equation}\label{Eq_4920_CommEst}
\| [ x^n, z ] \|
  \leq \sum_{k = 1}^n
    \| x^{n - k} \| \cdot \| [x, z] \| \cdot \| x^{k - 1} \|
  \leq n \| [ x, z ] \|.
\end{equation}

Choose $n \in \Nz$
and a polynomial $g (\ld) = \sum_{k = 0}^{n} \af_k \ld^k$
with coefficients $\af_k \in \C$
for $k = 0, 1, \ldots, n$
such that
$| g (\ld) - f (\ld) | < \tfrac{\ep}{3}$
for all $\ld \in [0, 1]$.
Define
\[
\dt = \frac{\ep}{1 + 3 \sum_{k = 1}^{n} k |\af_k|}.
\]
Suppose $\| [x, z] \| < \dt$.
Then, since $\| z \| \leq 1$
and using~(\ref{Eq_4920_CommEst})
at the second step,
\[
\| [f (x), z ] \|
  \leq 2 \| f (x) - g (x) \|
    + \sum_{k = 1}^{n} |\af_k|
             \cdot \| [ x^k, z ] \|
  < 2 \left( \frac{\ep}{3} \right)
         + \sum_{k = 1}^{n} k |\af_k| \dt
  < \ep.
\]
This completes the proof.
\end{proof}

\begin{lem}\label{L_4917_Near}
Suppose $f \colon [0, 1] \to \C$ is continuous
and $f (0) = 0$.
Then for every $\ep > 0$
there exists $\delta > 0$
such that whenever $D$ is a \ca,
$E \subset D$ is a subalgebra,
and $a \in D$ and $b \in E$
satisfy
\[
\| a \| \leq 1,
\,\,\,\,\,\,
0 \leq b \leq 1,
\andeqn
\dist (b a, E) < \delta,
\]
then $\dist (f (b) a, \, E) < \ep$.
\end{lem}

\begin{proof}
Choose $n \in \Nz$
and a polynomial $g (\ld) = \sum_{k = 0}^{n} \af_k \ld^k$
with $\af_k \in \C$
for $k = 0, 1, \ldots, n$
such that
$\af_0 = 0$
and
$| g (\ld) - f (\ld) | < \tfrac{\ep}{3}$
for all $\ld \in [0, 1]$.
Define
\[
\dt
 = \frac{\ep}{2 \big( 1 + |\af_1| + |\af_2| + \cdots + |\af_n| \big)}.
\]
Let $a \in D$ and $b \in E$ be as in the hypotheses.
Choose $y \in E$ such that $\| b a - y \| < \dt$.
Set $z = \sum_{k = 1}^n \af_k b^{k - 1} y \in E$.
Then
\[
\| f (b) a - z \|
  \leq \| f (b) - g (b) \| \cdot \| a \|
    + \sum_{k = 1}^{n} |\af_k|
             \cdot \| b^{k - 1} \| \cdot \| b a - y \|
  < \frac{\ep}{2} + \sum_{k = 1}^{n} |\af_k| \dt
  < \ep.
\]
This completes the proof.
\end{proof}

The following corollary is also easy to prove directly.

\begin{cor}\label{C_4926_InSub}
Let $D$ be a \ca,
let $E \subset D$ be a subalgebra,
let $a \in D$,
let $b \in E_{+}$,
and suppose that $b a \in E$.
Let $f \colon \spec (b) \to \C$ be a continuous function
such that $f (0) = 0$.
Then $f (b) a \in E$.
\end{cor}

\begin{proof}
By scaling,
\wolog{} $\| b \| \leq 1$.
Then we may extend $f$ to be a
\cfn{} $f \colon [0, 1] \to \C$.
Lemma~\ref{L_4917_Near} now shows that
for all $\ep > 0$, we have $\dist (f (b) a, \, E) < \ep$.
\end{proof}

\section{Definitions and preliminary lemmas}\label{Sec:LgDef}

\indent
In this section,
we recall the definition of a large subalgebra,
and give the definition of a centrally large subalgebra.
We then give some convenient variants of the definition.

Recall that a unital subalgebra of a \uca~$A$
is, by convention, assumed to contain the identity of~$A$.

The definitions of ``large'' and ``stably large''
in Definition~\ref{D_6213_Large}
originally appeared in
Definition~4.1 and Definition~5.1 of~\cite{Ph40}.
The definitions of ``centrally large'' and ``stably centrally large''
differ by the addition of an approximate commutation condition.

\begin{dfn}\label{D_6213_Large}
Let $A$ be an in\fd{} simple unital \ca.
A unital subalgebra $B \subset A$ is said to be
{\emph{large}} in~$A$ if
for every $m \in \N$,
$a_1, a_2, \ldots, a_m \in A$,
$\ep > 0$, $x \in A_{+}$ with $\| x \| = 1$,
and $y \in B_{+} \SM \{ 0 \}$,
there are $c_1, c_2, \ldots, c_m \in A$ and $g \in B$
such that:
\begin{enumerate}
\item\label{D_Large_Cut1}
$0 \leq g \leq 1$.
\item\label{D_Large_Cut2}
For $j = 1, 2, \ldots, m$ we have
$\| c_j - a_j \| < \ep$.
\item\label{D_Large_Cut3}
For $j = 1, 2, \ldots, m$ we have
$(1 - g) c_j \in B$.
\item\label{D_Large_Cut4}
$g \subeq_B y$ and $g \subeq_A x$.
\item\label{D_Large_Cut5}
$\| (1 - g) x (1 - g) \| > 1 - \ep$.
\setcounter{TmpEnumi}{\value{enumi}}
\end{enumerate}
We say that
$B$ is {\emph{centrally large}} in~$A$ if we can require that
in addition:
\begin{enumerate}
\setcounter{enumi}{\value{TmpEnumi}}
\item\label{D_Large_Cut6}
For $j = 1, 2, \ldots, m$ we have
$\| g a_j - a_j g \| < \ep$.
\end{enumerate}
We further say that
$B$ is {\emph{stably large}} in~$A$ if
$M_n (B)$ is large in $M_n (A)$ for all $n \in \Nz$,
and {\emph{stably centrally large}} in~$A$ if
$M_n (B)$ is centrally large in $M_n (A)$ for all $n \in \Nz$.
\end{dfn}

We emphasize that one of the Cuntz subequivalences
in~(\ref{D_Large_Cut4})
is relative to~$B$, not to~$A$.

By Proposition~4.5 of~\cite{Ph40},
when $A$ is finite
we can omit the condition involving $\| (1 - g) x (1 - g) \|$
in the definition of a large subalgebra.
The following proposition
gives the same result for \Rql{} subalgebras.

\begin{prp}\label{P-2729AltRq}
Let $A$ be an in\fd{} simple unital \ca,
and let $B \subset A$ be a unital subalgebra.
Assume that $A$ is finite and that
for every $m \in \N$,
$a_1, a_2, \ldots, a_m \in A$,
$\ep > 0$,
and $y \in B_{+} \setminus \{ 0 \}$,
there are $c_1, c_2, \ldots, c_m \in A$ and $g \in B$
such that:
\begin{enumerate}
\item\label{P-Rq:AltCut1}
$0 \leq g \leq 1$.
\item\label{P-Rq:AltCut2}
For $j = 1, 2, \ldots, m$ we have
$\| c_j - a_j \| < \ep$.
\item\label{P-Rq:AltCut3}
For $j = 1, 2, \ldots, m$ we have
$(1 - g) c_j \in B$.
\item\label{P-Rq:AltCut4}
$g \subeq_B y$ and $g \subeq_A x$.
\item\label{P-Rq:AltCut5}
For $j = 1, 2, \ldots, m$ we have
$\| g a_j - a_j g \| < \ep$.
\end{enumerate}
Then $B$ is a \Rql{} subalgebra of~$A$.
\end{prp}

\begin{proof}
The proof is the same
as that of Proposition~4.5 of~\cite{Ph40}.
\end{proof}

The following lemma states that the elements $c_j$
in the definition of a \Rql{} subalgebra
can be chosen with norm no larger than the norm of corresponding $a_j$.

\begin{lem}\label{L:LargeDecNorm}
Suppose $B \subset A$ is a \Rql{} subalgebra of $A$.
The elements
$c_1, c_2, \ldots, c_m$ may be chosen so that
$\| c_j \| \leq \| a_j \|$ for $j = 1, 2, \ldots, m$.
\end{lem}

\begin{proof}
The proof is the same as the proof of Lemma~4.7 in~\cite{Ph40}.
\end{proof}

If we cut down on both sides instead of
on one side,
and the elements $a_j$ are positive,
then we may take the elements $c_j$ to be positive.
We can
still choose $c_j$ with
norm no larger than that of~$a_j$.

\begin{lem}\label{L:LargeStaysPositive}
Let $A$ be an in\fd{} simple unital \ca.
Suppose $B \subset A$ is a centrally large subalgebra.
Let $m, n \in \Nz$,
let $a_1, a_2, \ldots, a_m \in A$,
let $b_1, b_2, \ldots, b_n \in A_{+}$,
let $\ep > 0$,
let $x \in A_{+}$ satisfy $\| x \| = 1$,
and let $y \in B_{+} \setminus \{ 0 \}$.
Then there are
$c_1, c_2, \ldots, c_m \in A$,
$d_1, d_2, \ldots, d_n \in A_{+}$,
and $g \in B$
such that:
\begin{enumerate}
\item\label{L:LargeStaysPositive:Cut1}
$0 \leq g \leq 1$.
\item\label{L:LargeStaysPositive:Cut2a}
For $j = 1, 2, \ldots, m$ we have
$\| c_j - a_j \| < \ep$,
and for $j = 1, 2, \ldots, n$ we have
$\| d_j - b_j \| < \ep$.
\item\label{L:LargeStaysPositive:Cut2b}
For $j = 1, 2, \ldots, m$ we have
$\| c_j \| \leq \| a_j \|$,
and for $j = 1, 2, \ldots, n$ we have
$\| d_j \| \leq \| b_j \|$.
\item\label{L:LargeStaysPositive:Cut3}
For $j = 1, 2, \ldots, m$ we have
$(1 - g) c_j \in B$,
and for $j = 1, 2, \ldots, n$ we have
$(1 - g) d_j (1 - g) \in B$.
\item\label{L:LargeStaysPositive:Cut4}
$g \subeq_B y$ and $g \subeq_A x$.
\item\label{L:LargeStaysPositive:Cut5}
$\| (1 - g) x (1 - g) \| > 1 - \ep$.
\item\label{L:LargeStaysPositive:Cut6}
$\| g a_j - a_j g \| < \ep$ for $j = 1, 2, \ldots, m$
and $\| g b_j - b_j g \| < \ep$ for $j = 1, 2, \ldots, n$.
\end{enumerate}
\end{lem}

\begin{proof}
The proof is the essentially the same as
the proof of Lemma~4.8 in~\cite{Ph40}.
\end{proof}

\begin{prp}\label{P_4920_TZLarge}
Let $A_1$ and $A_2$ be infinite dimensional simple unital \ca{s},
and let $B_1 \subset A_1$ and $B_2 \subset A_2$ be
centrally large subalgebras.
Assume that $A_1 \otimes_{\min} A_2$
is finite.
Then $B_1 \otimes_{\min} B_2$
is a centrally large subalgebra of $A_1 \otimes_{\min} A_2$.
\end{prp}

\begin{proof}
The proof is essentially the same as that of
Proposition~5.6 of~\cite{Ph40};
we describe the differences.

In place of Proposition~4.5 of~\cite{Ph40},
we use Proposition~\ref{P-2729AltRq}.
The analog for centrally large subalgebras
of Lemma~4.2 of~\cite{Ph40}
is obvious.

We also need to check the approximate commutation relation
in the conclusion,
that is,
using the notation of the proof of Proposition~5.6 of~\cite{Ph40},
we take
\[
g = 1 - (1 - g_1) \otimes (1 - g_2),
\]
and we need to prove that
\[
\| [ g, \, a_{1, j} \otimes a_{2, j} ] \| < \ep
\]
for $j = 1, 2, \ldots, m$.
For $l = 1, 2$,
since $B_l$ is centrally large in~$A_l$,
we may require,
in addition to the conditions demanded
in the proof of Proposition~5.6 of~\cite{Ph40},
that
$\| [g_l, a_{l, j} ] \| < \frac{\ep}{2}$.
Then also
$\| [1 - g_l, \, a_{l, j} ] \| < \frac{\ep}{2}$.
Therefore, using $\| 1 - g_l \| \leq 1$ and $\| a_{l, j} \| \leq 1$
at the fourth step,
we get
\begin{align*}
\| [ g, \, a_{1, j} \otimes a_{2, j} ] \|
& = \| [1 - g, \, a_{1, j} \otimes a_{2, j} ] \|
  = \big\| \big[ (1 - g_1) \otimes (1 - g_2),
                 \, a_{1, j} \otimes a_{2, j} \big] \big\|
\\
& \leq \big\| [1 - g_1, \, a_{1, j} ] \otimes (1 - g_2) a_{2, j} \big\|
     + \big\| a_{1, j} (1 - g_1) \otimes [1 - g_2, \, a_{2, j} ] \big\|
\\
& < \frac{\ep}{2} + \frac{\ep}{2}
= \ep.
\end{align*}
This completes the proof.
\end{proof}

We now get the analog of Corollary~5.8 of~\cite{Ph40}.

\begin{cor}\label{C_4920_StFinStZLg}
Let $A$ be a stably finite infinite dimensional simple unital \ca,
and let $B \subset A$ be a centrally large subalgebra.
Then $B$ is stably centrally large in~$A$.
\end{cor}

\begin{proof}
In Proposition~\ref{P_4920_TZLarge}
take $A_1 = B_1 = M_n$,
$A_2 = A$, and $B_2 = B$.
\end{proof}

In the definition of a centrally large subalgebra,
it suffices to verify the conditions
for $a_1, a_2, \ldots, a_m$
in a generating set.
(We do not know the analog of this statement for large subalgebras.
The approximate commutation relation
makes the difference.)
We do not use this result in the rest of the paper,
but it seems potentially useful elsewhere.

The proof requires the analog of Proposition~4.4 of~\cite{Ph40}.

\begin{lem}\label{L_4917_ZDist}
Let $A$ be an infinite dimensional simple unital \ca,
and let $B \subset A$ be a unital subalgebra.
Suppose that every finite set $F \subset A$,
$\ep > 0$, $x \in A_{+}$ with $\| x \| = 1$,
and $y \in B_{+} \setminus \{ 0 \}$,
there is $g \in B$
such that:
\begin{enumerate}
\item\label{L_4917_ZDist_Cut1}
$0 \leq g \leq 1$.
\item\label{L_4917_ZDist_Cut2}
$\dist ( (1 - g) a, \, B) < \ep$ for all $a \in F$.
\item\label{L_4917_ZDist_Cut4}
$g \subeq_B y$ and $g \subeq_A x$.
\item\label{L_4917_ZDist_Cut5}
$\| (1 - g) x (1 - g) \| > 1 - \ep$.
\item\label{L_4917_ZDist_Cut6}
$\| g a - a g \| < \ep$ for all $a \in F$.
\end{enumerate}
Then $B$ is centrally large in~$A$.
\end{lem}

\begin{proof}
The proof is essentially the same as
that of Proposition~4.4 of~\cite{Ph40},
except with one addition.
We describe only the addition.

We verify the conditions of Definition~\ref{D_6213_Large}.
Let $m \in \Nz$,
let $a_1, a_2, \ldots, a_m \in A$,
let $\ep > 0$,
let $x \in A_{+}$ satisfy $\| x \| = 1$,
and let $y \in B_{+} \setminus \{ 0 \}$.
Let the functions $f_0, f_1, f_2 \colon [0, 1] \to [0, 1]$
be as in the proof of Proposition~4.4 of~\cite{Ph40}.
Define a \cfn{} $f \colon [0, 1] \to [0, 1]$
by $f (\ld) = 1 - f_0 (1 - \ld)$
for $\ld \in [0, 1]$.
Apply Lemma~\ref{FuncCalcCom} with this function~$f$
and with $\ep$ as given,
getting $\dt > 0$.
Apply the hypothesis with $F = \{ a_1, a_2, \ldots, a_m \}$
and with $\min \left( \dt, \frac{\ep}{3} \right)$ in place of~$\ep$,
getting $g_0 \in B$.
(The difference is that in~\cite{Ph40} we used $\frac{\ep}{3}$ here.)
Define $r_0 = 1 - g_0$.
Set $g = 1 - f_2 (r_0)$.

For $j = 1, 2, \ldots, m$,
define $c_j$ as in the proof of Proposition~4.4 of~\cite{Ph40}.
The verification of
(\ref{D_Large_Cut1}),
(\ref{D_Large_Cut2}),
(\ref{D_Large_Cut3}),
(\ref{D_Large_Cut4}),
and~(\ref{D_Large_Cut5})
of Definition~\ref{D_6213_Large} is exactly as in~\cite{Ph40}.
To verify Definition~\ref{D_6213_Large}(\ref{D_Large_Cut6}),
for $j = 1, 2, \ldots, m$ we use $\| g_0 a_j - a_j g_0 \| < \dt$,
$\| a_j \| \leq 1$,
and the choice of~$\dt$,
to get $\| g a_j - a_j g \| < \ep$.
\end{proof}

\begin{prp}\label{P_4917_ZGen}
Let $A$ be an infinite dimensional simple unital \ca,
and let $B \subset A$ be a unital subalgebra.
Let $S \subset A$ be a subset which generates $A$ as a \ca.
Suppose that every finite set $F \subset S$,
$\ep > 0$, $x \in A_{+}$ with $\| x \| = 1$,
and $y \in B_{+} \setminus \{ 0 \}$,
there is $g \in B$
such that the conditions
(\ref{L_4917_ZDist_Cut1}),
(\ref{L_4917_ZDist_Cut2}),
(\ref{L_4917_ZDist_Cut4}),
(\ref{L_4917_ZDist_Cut5}),
and~(\ref{L_4917_ZDist_Cut6}) of Lemma~\ref{L_4917_ZDist} hold.
Then $B$ is centrally large in~$A$.
\end{prp}

\begin{proof}
It is clear that if $T \subset A$ spans a dense subset of~$A$,
then Lemma~\ref{L_4917_ZDist}
is still valid if in the hypotheses
we restrict to finite sets $F \subset T$.

Now suppose that $S \subset A$
generates $A$ as a norm closed subalgebra
(without using the adjoint),
and the hypotheses of the proposition hold for this set~$S$.
We prove that $B$ is centrally large in~$A$
by verifying the hypotheses of
the generalization of Lemma~\ref{L_4917_ZDist}
as in the first paragraph.
\Wolog{} $\| a \| \leq 1$ for all $a \in S$.
If the hypotheses hold for~$S$,
they also hold for $S \cup \{ 1 \}$,
so we may assume $1 \in S$.

Let $T$ be the set of all finite products of elements of~$S$.
Following the first paragraph,
let $F \subset T$ be finite,
let $\ep > 0$, let $x \in A_{+}$ satisfy $\| x \| = 1$,
and let $y \in B_{+} \setminus \{ 0 \}$;
it is enough to show that there is $g \in B$
such that the conditions
(\ref{L_4917_ZDist_Cut1}),
(\ref{L_4917_ZDist_Cut2}),
(\ref{L_4917_ZDist_Cut4}),
(\ref{L_4917_ZDist_Cut5}),
and~(\ref{L_4917_ZDist_Cut6}) of Lemma~\ref{L_4917_ZDist} hold.
\Wolog{} $\ep < 1$.

Since $1 \in S$,
there are $n \in \N$
and a finite set $E \subset S$ such that $F$ is contained in
the finite set
\[
F_0
 = \big\{ a_1 a_2 \cdots a_n \colon a_1, a_2, \ldots, a_n \in E \big\}.
\]
We verify the conditions
of Lemma~\ref{L_4917_ZDist}
for $F_0$ instead of~$F$.

Use Lemma~\ref{FuncCalcCom}
to choose $\dt_1 > 0$ such that
whenever $g \in A_{+}$ and $a \in A$ satisfy
\[
0 \leq g \leq 1,
\,\,\,\,\,\,
\| a \| \leq 1,
\andeqn
\| g a - a g \| < \delta_1,
\]
then
\begin{equation}\label{Eq_4920_PwCm}
\big\| (1 - g)^{1/n} a - a (1 - g)^{1/n} \big\| < \frac{\ep}{n^2}.
\end{equation}
Use Lemma~\ref{L_4917_Near}
(taking $b = 1 - g)$)
to choose $\dt_2 > 0$ such that
whenever $g \in B_{+}$ and $a \in A$ satisfy
\[
\| a \| \leq 1,
\,\,\,\,\,\,
0 \leq g \leq 1,
\andeqn
\dist ((1 - g) a, \, B) < \delta_2,
\]
then
\[
\dist \big( (1 - g)^{1/n} a, \, B \big) < \frac{\ep}{2^n n}.
\]
Set $\dt = \min \big( \dt_1, \dt_2, \frac{\ep}{n} \big)$.
Apply the assumption
with $E$ in place of~$F$,
with $\dt$ in place of~$\ep$,
and with $x$ and $y$ as given,
obtaining $g \in B$.

Conditions (\ref{L_4917_ZDist_Cut1}),
(\ref{L_4917_ZDist_Cut4}),
and~(\ref{L_4917_ZDist_Cut5}) are immediate.
For~(\ref{L_4917_ZDist_Cut6}),
let $a_1, a_2, \ldots, a_n \in E$.
We use
$\| g a_j - a_j g \| < \frac{\ep}{n}$
and $\| a_j \| \leq 1$ for $j = 1, 2, \ldots, n$
to get
\begin{equation}\label{Eq_4920_Comm}
\| [ g, \, a_1 a_2 \cdots a_n ] \|
 \leq \sum_{j = 1}^n \| a_1 a_2 \cdots a_{j - 1} \|
      \cdot \| g a_j - a_j g \|
      \cdot \| a_{j + 1} a_{j + 2} \cdots a_n \|
 < \ep.
\end{equation}

We verify~(\ref{L_4917_ZDist_Cut2}).
Using~(\ref{Eq_4920_PwCm}),
an estimate similar to that for~(\ref{Eq_4920_Comm}) gives
\[
\big\| \big[ (1 - g)^{1/n}, \, a_1 a_2 \cdots a_k \big]
               a_{k + 1} a_{k + 2} \cdots a_n \big\|
    < \frac{k \ep}{n^2}
\]
for $k = 1, 2, \ldots, n$.
Using $( (1 - g)^{1/n} )^n = 1 - g$
and
\[
\| 1 - g \|, \, \| a_1 \|, \, \| a_2 \|, \, \ldots, \, \| a_n \| \leq 1,
\]
we then get
\begin{equation}\label{Eq_4920_Full}
\big\| (1 - g) a_1 a_2 \cdots a_n
   - (1 - g)^{1/n} a_2 (1 - g)^{1/n} a_1 \cdots (1 - g)^{1/n} a_n \big\|
 < \sum_{k = 0}^{n - 1} \frac{k \ep}{n^2}
 < \frac{\ep}{2}.
\end{equation}
Now for $j = 1, 2, \ldots, n$,
use $\dist ( (1 - g) a_j, \, B) < \dt_2$ and the choice of $\dt_2$
to choose $b_j \in B$ such that
\[
\| (1 - g)^{1/n} a_j - b_j \| < \frac{\ep}{2^n n}.
\]
Since $\ep < 1$ and $\| a_j \| \leq 1$,
we have $\| b_j \| \leq 2$;
also,
$\| (1 - g)^{1/n} a_j \| \leq 1 \leq 2$.
For $k = 1, 2, \ldots, n$,
we therefore get
\begin{align*}
& \big\| b_1 b_2 \cdots b_{k - 1}
      \big[ (1 - g)^{1/n} a_k - b_k \big]
          (1 - g)^{1/n} a_{k + 1} (1 - g)^{1/n} a_{k + 2}
           \cdots (1 - g)^{1/n} a_n \big\|
\\
& \alignInd
 < 2^{n - 1} \left( \frac{\ep}{2^n n} \right)
 = \frac{\ep}{2 n}.
\end{align*}
It follows that
\[
\big\| b_1 b_2 \cdots b_n
  - (1 - g)^{1/n} a_1 (1 - g)^{1/n} a_2 (1 - g)^{1/n} a_n \big\|
 < \frac{\ep}{2}.
\]
Since $b_1 b_2 \cdots b_n \in B$,
combining this estimate with~(\ref{Eq_4920_Full})
gives
\[
\dist \big( (1 - g) a_1 a_2 \cdots a_n, \, B \big) < \ep.
\]
This completes the proof of~(\ref{L_4917_ZDist_Cut2}),
and the verification of the proposition
under the assumption that $S$
generates $A$ as a norm closed subalgebra.

We now prove the proposition as stated by reducing the general case
to the case just done.
\Wolog{} $\| a \| \leq 1$ for all $a \in S$.
Set $T = S \cup S^*$.
Let $\ep > 0$,
let $F \subset T$  be finite,
let $\ep > 0$, let $x \in A_{+}$ satisfy $\| x \| = 1$,
and let $y \in B_{+} \setminus \{ 0 \}$.
Choose $g$ as in the hypotheses with $\frac{\ep}{2}$
in place of~$\ep$.
Conditions (\ref{L_4917_ZDist_Cut1}),
(\ref{L_4917_ZDist_Cut4}),
and~(\ref{L_4917_ZDist_Cut5}) are immediate.
Conditions (\ref{L_4917_ZDist_Cut2})
and~(\ref{L_4917_ZDist_Cut6})
for elements of $S$ are immediate.
So let $a \in S^*$.
Thus $a^* \in S$,
so $g^* = g$ implies
\[
\| g a - a g \|
  = \| a^* g - g a^* \|
  < \frac{\ep}{2}
  < \ep,
\]
which is~(\ref{L_4917_ZDist_Cut6}).
Also,
\[
\dist (a g, B)
 = \dist ( (a g)^*, \, B)
 = \dist (g a^*, B)
 < \frac{\ep}{2},
\]
so
\[
\dist (g a, B)
 \leq \| g a - a g \| + \dist (a g, B)
 < \frac{\ep}{2} +  \frac{\ep}{2}
 = \ep,
\]
which is~(\ref{L_4917_ZDist_Cut2}).
\end{proof}

\section{Large subalgebras of crossed product type}\label{Sec:Rq}

\indent
The motivating example for centrally large subalgebras
(Definition~\ref{D_6213_Large})
is subalgebras which arise in the study of crossed products
by minimal homeomorphisms.
In particular,
examples of centrally large subalgebras
are provided by orbit breaking subalgebras $C^* ({\mathbb{Z}}, X, h)_Y$
as in the introduction,
under the condition that $h^n (Y) \cap Y = \varnothing$
for all $n \in \Z \setminus \{ 0 \}$.
These are generalizations of Putnam's algebra
$C^* ({\mathbb{Z}}, X, h)_{ \{ y \} }
 \subset C^* ({\mathbb{Z}}, X, h)$
from~\cite{Pt1},
and the condition above,
that $Y$ meet each orbit at most once,
was used by Putnam in Example~2.6 of~\cite{Pt4}.
See the discussion after Definition~7.3 in~\cite{Ph40}
for more,
and for further generalizations.
These are in fact large subalgebras of crossed product type
(Definition~\ref{D-2717CPType} below).
Definition~\ref{D-2717CPType}
has the advantage of not explicitly
requiring any commutation relations.

\begin{dfn}[Definition~4.9 of~\cite{Ph40}]\label{D-2717CPType}
Let $A$ be
an in\fd{} simple separable unital \ca.
A subalgebra $B \subset A$ is said to be a
{\emph{large subalgebra of crossed product type}}
if there exist a subalgebra
$C \subset B$ and a subset $G$ of the unitary group of~$A$ such that:
\begin{enumerate}
\item\label{D-2717CPType-Sb}
\begin{enumerate}
\item\label{D-2717CPType-Sb1}
$C$ contains the identity of~$A$.
\item\label{D-2717CPType-Sb2}
$C$ and $G$ generate $A$ as a \ca.
\item\label{D-2717CPType-Sb3}
$u C u^* \subset C$ and $u^* C u \subset C$ for all $u \in G$.
\end{enumerate}
\item\label{D-2717CPType-Cut}
For every $m \in \N$,
$a_1, a_2, \ldots, a_m \in A$,
$\ep > 0$, $x \in A_{+}$ with $\| x \| = 1$,
and $y \in B_{+} \setminus \{ 0 \}$,
there are $c_1, c_2, \ldots, c_m \in A$ and $g \in C$
such that:
\begin{enumerate}
\item\label{D-2717CPType-Cut1}
$0 \leq g \leq 1$.
\item\label{D-2717CPType-Cut2}
For $j = 1, 2, \ldots, m$ we have
$\| c_j - a_j \| < \ep$.
\item\label{D-2717CPType-Cut3}
For $j = 1, 2, \ldots, m$ we have
$(1 - g) c_j \in B$.
\item\label{D-2717CPType-Cut4}
$g \subeq_B y$ and $g \subeq_A x$.
\item\label{D-2717CPType-Cut5a}
$\| (1 - g) x (1 - g) \| > 1 - \ep$.
\end{enumerate}
\end{enumerate}
\end{dfn}

The conditions in~(\ref{D-2717CPType-Cut})
are the same as the conditions in Definition~\ref{D_6213_Large};
the difference is that
we require that $g \in C$,
not merely that $g \in B$.
In particular,
we have the following relation between the properties.

\begin{prp}[Proposition~4.10 of~\cite{Ph40}]\label{P-2814CTImpL}
Let $A$ be an in\fd{} simple separable unital \ca,
and let $B \subset A$ be a subalgebra.
If $B$ is a large subalgebra of~$A$
of crossed product type
in the sense of Definition~\ref{D-2717CPType},
then $B$ is a large subalgebra of~$A$
in the sense of Definition~\ref{D_6213_Large}.
\end{prp}

\begin{proof}
This is immediate from the definitions.
\end{proof}

In the basic examples,
$G$ will be a discrete amenable group,
$A$ will have the form $A = C^* (G, X)$
for some essentially free minimal action of $G$ on a \cms~$X$,
the set $G$ in the definition will be the standard
copy of the group $G$ in $C^* (G, X)$,
the subalgebra $C$ will be $C (X)$,
and $B$ will be a subalgebra of~$A$ which contains $C (X)$
but is,
informally,
much closer to $A$ than to $C (X)$.
However,
in the definition we do not require
that $C$ be commutative,
and we do not require that $G$ be a group.
In particular,
the definition may be applicable to crossed products
of noncommutative \ca{s},
to cocycle crossed products,
and to quotients of some nonsimple crossed products.
However,
something more complicated may well be needed for use
with groupoid \ca{s}
or more general crossed products.

The main result of this section is that large subalgebras of
crossed product type are \Rql.
To improve readability,
we isolate several parts of the argument as separate lemmas.

\begin{lem}\label{L-2720Narrow}
Let $A$ be a \ca,
let $\ep, \rh \geq 0$,
and let $a, b, d, r, x, y \in A_{\sa}$ satisfy the following:
\[
a \leq d \leq b,
\,\,\,\,\,\,
\| b - a \| \leq \ep,
\,\,\,\,\,\,
\| a - x \| \leq \rh,
\,\,\,\,\,\,
\| r - y \| \leq \rh,
\,\,\,\,\,\,
x \leq r,
\andeqn
y \leq b.
\]
Then $\| d - r \| \leq \ep + \rh$.
\end{lem}

\begin{proof}
Subtracting $d$ from all the elements listed,
we reduce to the case $d = 0$,
and we must prove that $\| r \| \leq \ep + \rh$.

Since $\| a - x \| \leq \rh$, we have $a - \rh \leq x$.
{}From $\| a - b \| \leq \ep$ and $b \geq 0$,
we get $- \ep \leq a - b \leq a$.
So $- \ep - \rh \leq a - \rh \leq x \leq r$.

{}From $\| r - y \| \leq \rh$ we get $r \leq y + \rh$.
Since $a \leq 0$ and $\| b - a \| \leq \ep$,
we have $y \leq b \leq b - a \leq \ep$.
Therefore $r \leq \ep + \rh$.
Combining this with the conclusion of
the previous paragraph,
we get $\| r \| \leq \ep + \rh$, as desired.
\end{proof}

\begin{lem}\label{L-2720RtAI}
For every $\ep > 0$ there is $\dt > 0$ such that
whenever $A$ is a \ca\  and $e, x \in A$ satisfy
\[
0 \leq e \leq 1,
\,\,\,\,\,\,
0 \leq x \leq 1,
\andeqn
\| e x - x \| < \dt,
\]
then $\big\| e^{1/2} x - x \big\| < \ep$.
\end{lem}

\begin{proof}
Choose a polynomial $p_0 (\ld) = \sum_{k = 0}^{n} \bt_k \ld^k$
with coefficients $\bt_k \in \R$
such that
$\big| p_0 (\ld) - \ld^{1/2} \big| < \tfrac{\ep}{8}$
for all $\ld \in [0, 1]$.
Define
\[
\af_0 = 0,
\,\,\,\,\,\,
\af_1 = 1 - \sum_{k = 2}^{n} \bt_k,
\,\,\,\,\,\,
\af_2 = \bt_2,
\,\,\,\,\,\,
\af_3 = \bt_3,
\,\,\,\,\,\,
\ldots,
\,\,\,\,\,\,
\af_n = \bt_n.
\]
Since $| p_0 (0) | < \tfrac{\ep}{8}$,
we have $| \af_0 - \bt_0 | < \tfrac{\ep}{8}$.
Since $| p_0 (1) - 1 | < \tfrac{\ep}{8}$,
we have
\[
| \af_1 - \bt_1 |
 = \left| 1 - \bt_1 - \sum_{k = 2}^{n} \bt_k \right|
 \leq \left| 1 - \sum_{k = 0}^{n} \bt_k \right| + | \bt_0 |
 = | 1 - p_0 (1) | + | \bt_0 |
 < \frac{2 \ep}{8}.
\]
Therefore,
for $\ld \in [0, 1]$,
the polynomial
$p (\ld) = \sum_{k = 0}^{n} \af_k \ld^k$
satisfies
\[
\big| p (\ld) - \ld^{1/2} \big|
 \leq | \af_0 - \bt_0 | + | \af_1 - \bt_1 | \ld
    + \big| p_0 (\ld) - \ld^{1/2} \big|
 < \frac{\ep}{8} + \frac{2 \ep}{8} + \frac{\ep}{8}
 = \frac{\ep}{2}.
\]
Moreover $\af_0 = 0$ and $\sum_{k = 0}^{n} \af_k = 1$.

Define
\[
\dt = \frac{\ep}{1 + 2 \sum_{k = 1}^{n} k | \af_k |}.
\]
Suppose now $\| e x - x \| < \dt$.
An induction argument,
using the estimate
\[
\| e^k x - x \|
 \leq \| e^{k - 1} \| \cdot \| e x - x \|
    + \| e^{k - 1} x - x \|,
\]
shows that for all $k \in \N$ we have
$\| e^k x - x \| < k \dt$.
Now,
using $\sum_{k = 1}^{n} \af_k x = x$ at the second step,
we get
\begin{align*}
\big\| e^{1/2} x - x  \big\|
& \leq \big\| e^{1/2} - p (e) \big\| \cdot \| x \|
         + \| p (e) x - x \|
\\
& \leq \big\| e^{1/2} - p (e) \big\|
         + \sum_{k = 1}^{n} | \af_k | \cdot \| e^k x - x \|
  < \frac{\ep}{2} + \sum_{k = 0}^{n} k | \af_k | \dt
  < \ep.
\end{align*}
This completes the proof.
\end{proof}

\begin{lem}\label{L-L2727Csub}
Let $A$ be a simple unital \ca,
let $B \subset A$ be a large subalgebra of crossed product type
(Definition~\ref{D-2717CPType}),
and let $C \subset B$
be the subalgebra of Definition~\ref{D-2717CPType}.
Assume $B \neq A$.
Then for every $y \in B_{+} \setminus \{ 0 \}$
there is $z \in C_{+} \setminus \{ 0 \}$ such that $z \subeq_B y$.
\end{lem}

\begin{proof}
Choose any $a \in A$ such that $\dist (a, B) > 2$.
Definition~\ref{D-2717CPType} gives
$c \in A$ and $g \in C$
such that
\[
0 \leq g \leq 1,
\,\,\,\,\,\,
\| c - a \| < 1,
\,\,\,\,\,\,
(1 - g) c \in B,
\andeqn
g \subeq_B y.
\]
The second condition implies that $c \not\in B$,
so the third condition implies that $g \neq 0$.
Therefore the lemma is proved by taking $z = g$.
\end{proof}

\begin{thm}\label{T-2717CptImpRq}
Let $A$ be a stably finite simple unital \ca,
and let $B \subset A$ be a large subalgebra of crossed product type
(Definition~\ref{D-2717CPType}).
Then $B$ is \Rql\  in~$A$ (Definition~\ref{D_6213_Large}).
\end{thm}

We hope that finiteness of $A$ is not necessary,
but we do not know how to prove this.
The difficulty is with condition~(\ref{D-2717RqPf-Cut5a})
in the proof;
when $A$ is finite,
this condition is, in effect, automatic.

We will use a weak
(and obvious) variant of Proposition~\ref{P_4917_ZGen}:
we only restrict to a generating set to verify the
approximate commutation condition.

\begin{proof}[Proof of Theorem~\ref{T-2717CptImpRq}]
For a finite set $T$ and $b_t \in A$ for $t \in T$,
we take $\bigoplus_{t \in T} b_t \in M_{\card (T)} (A)$
to be the diagonal matrix with diagonal entries $b_t$ for $t \in T$.
(The order won't matter in this proof.)
In particular, $\bigoplus_{t \in T} b$ is a $\card (T) \times \card (T)$
diagonal matrix with all entries equal to~$b$.

If $B = A$,
then the conclusion is trivial.
(We can always take $g = 0$.)
We may therefore assume that $B \neq A$.

We must prove that for every $m \in \N$,
$a_1, a_2, \ldots, a_m \in A$,
$\ep > 0$, $x \in A_{+}$ with $\| x \| = 1$,
and $y \in B_{+} \setminus \{ 0 \}$,
there are $c_1, c_2, \ldots, c_m \in A$ and $g \in B$
such that:
\begin{enumerate}
\item\label{D-2717RqPf-Cut1}
$0 \leq g \leq 1$.
\item\label{D-2717RqPf-Cut2}
For $j = 1, 2, \ldots, m$ we have
$\| c_j - a_j \| < \ep$.
\item\label{D-2717RqPf-Cut3}
For $j = 1, 2, \ldots, m$ we have
$(1 - g) c_j \in B$.
\item\label{D-2717RqPf-Cut4}
$g \subeq_B y$ and $g \subeq_A x$.
\item\label{D-2717RqPf-Cut5a}
$\| (1 - g) x (1 - g) \| > 1 - \ep$.
\item\label{D-2717RqPf-Cut5}
For $j = 1, 2, \ldots, m$ we have
$\| g a_j - a_j g \| < \ep$.
\setcounter{TmpEnumi}{\value{enumi}}
\end{enumerate}
It is clearly equivalent to choose some fixed set $S \subset A$
which generates $A$ as a \ca,
and then to prove that for every $m \in \N$,
$a_1, a_2, \ldots, a_m \in A$,
finite subset $F \subset S$,
$\ep > 0$, $x \in A_{+}$ with $\| x \| = 1$,
and $y \in B_{+} \setminus \{ 0 \}$,
there are $c_1, c_2, \ldots, c_m \in A$ and $g \in B$
such that
(\ref{D-2717RqPf-Cut1}),
(\ref{D-2717RqPf-Cut2}),
(\ref{D-2717RqPf-Cut3}),
(\ref{D-2717RqPf-Cut4}),
and~(\ref{D-2717RqPf-Cut5a})
above hold,
and such that in place of~(\ref{D-2717RqPf-Cut5})
the following holds:
\begin{enumerate}
\setcounter{enumi}{\value{TmpEnumi}}
\item\label{D-2717RqMod-Cut5}
For every $x \in F$ we have
$\| g x - x g \| < \ep$.
\setcounter{TmpEnumi}{\value{enumi}}
\end{enumerate}

Now let $C \subset B$ and $G \subset A$
be as in Definition~\ref{D-2717CPType}.
We take the set $S$ above to be
the union of $G$ and the unitary group of~$C$.
Then $S$ is a set of unitaries which generates $A$ as a \ca,
and $u C u^* \subset C$ and $u^* C u \subset C$ for all $u \in S$.

Let $m \in \N$,
let $a_1, a_2, \ldots, a_m \in A$,
let $F \subset S$ be finite,
let $\ep > 0$,
let $x \in A_{+}$ satisfy $\| x \| = 1$,
and let $y \in B_{+} \setminus \{ 0 \}$.
Choose $n \in \N$ such that
\begin{equation}\label{Eq_6213_1ne8}
\frac{1}{n} < \frac{\ep}{8}.
\end{equation}
Set
\[
F_0 = \{ 1 \},
\,\,\,\,\,\,
F_1 = F \cup F^* \cup \{ 1 \},
\,\,\,\,\,\,
F_2 = F_1 F_1,
\,\,\,\,\,\,
F_3 = F_1 F_2,
\,\,\,\,\,\,
\ldots,
\,\,\,\,\,\,
F_n = F_1 F_{n - 1}.
\]
Set $N = \card (F_1)$.
Use Lemma~2.9 of~\cite{Ph40}
to choose $y_0 \in ( {\overline{x A x}})_{+} \setminus \{ 0 \}$
such that:
\begin{enumerate}
\setcounter{enumi}{\value{TmpEnumi}}
\item\label{Cd_6213_y0}
If $g \in A_{+}$ satisfies $0 \leq g \leq 1$
and $g \subeq_A y_0$,
then $\| (1 - g) x (1 - g) \| > 1 - \ep$.
\setcounter{TmpEnumi}{\value{enumi}}
\end{enumerate}
In particular,
\begin{equation}\label{Eq_6213_y0Ax}
y_0 \subeq_A x.
\end{equation}
Proposition~\ref{P-2814CTImpL}
implies that $B$ is a large subalgebra of~$A$,
and Proposition~5.2 of~\cite{Ph40}
now implies that $B$ is simple.
Use Lemma~5.3 of~\cite{Ph40}
with $r = 1$, $a = y_0$, and $\ep = 1$
to choose
$y_1 \in B_{+} \setminus \{ 0 \}$
such that
\begin{equation}\label{Eq_6213_y1Ay0}
y_1 \subeq_A y_0.
\end{equation}
(We don't need part (3) of the conclusion of this lemma,
and part~(1) is only used to ensure that $y_1 \neq 0$.)
Use simplicity of~$B$
and Lemma~2.6 of~\cite{Ph40}
to choose $y_2 \in B_{+} \setminus \{ 0 \}$
such that
\begin{equation}\label{Eq_6213_y2By1y2By}
y_2 \subeq_B y_1
\andeqn
y_2 \subeq_B y.
\end{equation}
Since $B$ is simple,
Lemma~2.4 of~\cite{Ph40}
provides a nonzero positive
element $y_3 \in {\overline{y_2 B y_2}}$
such that the following condition holds:
\begin{enumerate}
\setcounter{enumi}{\value{TmpEnumi}}
\item\label{Cd_6213_ManyOrth}
There are $N^{n - 1} + 1$ orthogonal positive elements
in ${\overline{y_2 B y_2}}$,
each Cuntz equivalent in~$B$ to~$y_3$.
\setcounter{TmpEnumi}{\value{enumi}}
\end{enumerate}
Now use Lemma~\ref{L-L2727Csub}
to choose $z_0 \in C_{+} \setminus \{ 0 \}$
such that
\begin{equation}\label{Eq_6213_z0By3}
z_0 \subeq_B y_3.
\end{equation}
Then $u^* z_0 u \in C \subset B$ for all $u \in F_{n - 1}$.
Use Lemma~2.6 of~\cite{Ph40}
(applied to~$B$) to choose $z_1 \in B_{+} \setminus \{ 0 \}$
such that
\begin{equation}\label{Eq_6213_z1Buz0}
z_1 \subeq_B u^* z_0 u
\end{equation}
for all $u \in F_{n - 1}$.
Then use Lemma~\ref{L-L2727Csub}
to choose $z \in C_{+} \setminus \{ 0 \}$
such that
\begin{equation}\label{Eq_6213_zBz1}
z \subeq_B z_1.
\end{equation}

Since $A$ is simple,
Lemma~2.4 of~\cite{Ph40}
provides a nonzero positive
element $x_0 \in {\overline{x A x}}$
such that the following condition holds:
\begin{enumerate}
\setcounter{enumi}{\value{TmpEnumi}}
\item\label{Cd_6213_FnOrth}
There are $\card (F_{n - 1})$ orthogonal positive elements
in ${\overline{x A x}}$,
each Cuntz equivalent in~$A$ to~$x_0$.
\setcounter{TmpEnumi}{\value{enumi}}
\end{enumerate}

Apply Definition~\ref{D-2717CPType}
with $a_1, a_2, \ldots, a_m$,
and $\ep$ as given,
and with $z$ in place of~$y$ and $x_0$ in place of $x$.
Let $c_1, c_2, \ldots, c_m \in A$ and $g_0 \in C$
be the resulting elements.
In particular,
\begin{equation}\label{Eq_6213_g0Bzg0Ax0}
g_0 \subeq_B z
\andeqn
g_0 \subeq_A x_0.
\end{equation}

Choose $\ep_0 > 0$ so small that the following three things happen:
\begin{enumerate}
\setcounter{enumi}{\value{TmpEnumi}}
\item\label{6214_Ep0_e8}
$\ep_0 < \tfrac{\ep}{8}$.
\item\label{6214_Ep0_FCC}
The choice $2 \ep_0$ can be used for~$\dt$
in Lemma~\ref{FuncCalcCom} when $\ep$ there
is taken to be~$\tfrac{\ep}{8}$ and when $f(x) = x^{1/2}$.
\item\label{6214_Ep0_RTAI}
The choice $\ep_0$ can be used for~$\dt$
in Lemma~\ref{L-2720RtAI} when $\ep$ there
is taken to be~$\tfrac{\ep}{8}$.
\setcounter{TmpEnumi}{\value{enumi}}
\end{enumerate}

For $k = 0, 1, \ldots, n$ and all $u \in F_k$,
we have $u g_0 u^* \in C$.
For $k = 0, 1, \ldots, n$,
let $D_k \subset B$ be the hereditary
subalgebra of~$B$ generated by all $u g_0 u^*$ for $u \in F_k$.
Then $u D_k u^* \subset D_{k + 1}$
for $u \in F_1$ and $k = 0, 1, \ldots, n - 1$.
By induction, choose
\[
g_1 \in D_1,
\,\,\,\,\,\,
g_2 \in D_2,
\,\,\,\,\,\,
\ldots,
\,\,\,\,\,\,
g_{n} \in D_{n}
\]
such that
\begin{equation}\label{Eq_6214_AppId}
0 \leq g_k \leq 1
\andeqn
\| g_k u g_l u^* - u g_l u^* \| < \ep_0
\end{equation}
for
\[
k = 1, 2, \ldots, n,
\,\,\,\,\,\,
u \in F_1,
\andeqn
l = 0, 1, \ldots, k - 1.
\]

Now define
\begin{equation}\label{Eq_6214_Dfg}
r_0
 = \left( 1 - \frac{1}{n} \sum_{k = 1}^{n - 1} g_k \right)^{1/2}
     \!\!\!  \cdot (1 - g_0)^{1/2},
\,\,\,\,\,\,
r = r_0^* r_0,
\andeqn
g = 1 - r.
\end{equation}
We verify conditions
(\ref{D-2717RqPf-Cut1}),
(\ref{D-2717RqPf-Cut2}),
(\ref{D-2717RqPf-Cut3}),
(\ref{D-2717RqPf-Cut4}),
(\ref{D-2717RqPf-Cut5a}),
and~(\ref{D-2717RqMod-Cut5}).

For~(\ref{D-2717RqPf-Cut1}),
we clearly have $\| r_0 \| \leq 1$,
so $0 \leq r \leq 1$,
whence $0 \leq g \leq 1$.

Condition~(\ref{D-2717RqPf-Cut2})
follows from the choice of $c_1, c_2, \ldots, c_m$.

We verify condition~(\ref{D-2717RqPf-Cut3}).
Corollary~\ref{C_4926_InSub} implies that $(1 - g_0)^{1/2} c_j \in B$
for $j = 1, 2, \ldots, m$.
Since
$(1 - g_0)^{1/2} \in B$
and $1 - \frac{1}{n} \sum_{k = 1}^{n - 1} g_k  \in B$,
it now follows that $r c_j \in B$
for $j = 1, 2, \ldots, m$.
This is~(\ref{D-2717RqPf-Cut3}).

Now we prove (\ref{D-2717RqPf-Cut4}) and~(\ref{D-2717RqPf-Cut5a}).
Set
\begin{equation}\label{Eq:2724DefH}
s = 1 - (1 - g_0)^{1/2}
\andeqn
h = \frac{1}{n} \sum_{k = 0}^{n - 1} g_k.
\end{equation}
Then $s = f (g_0)$ for the \cfn\  $f \colon [0, 1] \to [0, 1]$
given by $f (\ld) = 1 - (1 - \ld)^{1/2}$.
Since $f (0) = 0$ and $g_0 \in D_0$,
it follows that $s \in D_0 \subset D_{n - 1}$.
Also
$h - \tfrac{1}{n} g_0 \in D_{n - 1}$.
Therefore
\begin{align*}
g
& = 1 - r
  = 1 - (1 - s) \big[ 1 - \big( h - \tfrac{1}{n} g_0 \big) \big] (1 - s)
\\
& = 2 s + h - \tfrac{1}{n} g_0
     - s \big( h - \tfrac{1}{n} g_0 \big) - s^2
     - \big( h - \tfrac{1}{n} g_0 \big) s
     + s \big( h - \tfrac{1}{n} g_0 \big) s
  \in D_{n - 1}.
\end{align*}
Now
$d = \sum_{u \in F_{n - 1}} u g_0 u^*$
satisfies ${\overline{d D_{n - 1} d}} = D_{n - 1}$.
In~$B$,
using Proposition 2.7(i) of~\cite{KR}
at the first step
and~(\ref{Eq_6213_g0Bzg0Ax0}) at the second step,
we now get
\[
g \subeq_B \sum_{u \in F_{n - 1}} u g_0 u^*
  \subeq_A \bigoplus_{u \in F_{n - 1}} u z u^*.
\]
For $u \in F_{n - 1}$, we have $z \subeq_B u^* z_0 u$
by (\ref{Eq_6213_zBz1}) and~(\ref{Eq_6213_z1Buz0}),
so $u z u^* \subeq_A z_0$.
Thus, using~(\ref{Eq_6213_z0By3}) at the last step,
\[
g \subeq_A \bigoplus_{u \in F_{n - 1}} u z u^*
  \subeq_A \bigoplus_{u \in F_{n - 1}} z_0
  \subeq_B \bigoplus_{u \in F_{n - 1}} y_3.
\]
Now
\[
y_3 \oplus \bigoplus_{u \in F_{n - 1}} y_3 \subeq_B y_2
\]
by~(\ref{Cd_6213_ManyOrth}).
Thus $y_3 \oplus g \subeq_A y_2$.
Corollary~5.8 of~\cite{Ph40}
implies that $B$ is stably large in~$A$.
So we can apply Lemma~6.5 of~\cite{Ph40},
with $a = g$,
$b = y_2$,
$c = \bigoplus_{u \in F_{n - 1}} y_3$,
and $x = y_3$,
to get $g \subeq_B y_2$.
Thus (\ref{Eq_6213_y2By1y2By}) gives $g \subeq_B y_2 \subeq_B y$.
By (\ref{Eq_6213_y2By1y2By})
and~(\ref{Eq_6213_y1Ay0})
we also get $g \subeq_B y_2 \subeq_B y_1 \subeq_A y_0$,
so~(\ref{Cd_6213_y0}) implies that
$\| (1 - g) x (1 - g) \| > 1 - \ep$,
and~(\ref{Eq_6213_y0Ax}) implies $g \subeq_A x$.

Finally,
we prove~(\ref{D-2717RqMod-Cut5}).
We first claim that
\begin{equation}\label{Eq_6213_SumCommEst}
\left\| \left[ 1 - g_0, \, 1 - \frac{1}{n} \sum_{k = 1}^{n - 1} g_k
                          \right] \right\|
     < 2 \ep_0.
\end{equation}
It is enough to prove that
\[
\left\| \left[ g_0, \, \frac{1}{n} \sum_{k = 1}^{n - 1} g_k
                          \right] \right\|
     < 2 \ep_0.
\]
The claim thus follows by using~(\ref{Eq_6214_AppId}) to get,
for $k = 1, 2, \ldots, n - 1$,
\[
\| g_k g_0 - g_0 g_k \|
  \leq \| g_k g_0 - g_0 \| + \| g_0 - g_0 g_k \|
  < \ep_0 + \ep_0
  = 2 \ep_0.
\]

Using (\ref{6214_Ep0_FCC}) and~(\ref{Eq_6213_SumCommEst}), we get
\[
\left\| \left[
  (1 - g_0)^{1/2}, \, 1 - \frac{1}{n} \sum_{k = 1}^{n - 1} g_k
                          \right] \right\|
     < \frac{\ep}{8}.
\]
Therefore
\begin{equation}\label{Eq:2724EstR}
\left\|
r - (1 - g_0) \left( 1 - \frac{1}{n} \sum_{k = 1}^{n - 1} g_k \right)
       \right\|
  < \frac{\ep}{8}.
\end{equation}

Recall the definition of~$h$ in~(\ref{Eq:2724DefH}).
We now claim that
\begin{equation}\label{Eq:2724gmh}
\| h - g \| < \frac{\ep}{4}.
\end{equation}
To prove this,
we use~(\ref{Eq:2724EstR}) and $g = 1 - r$ (from~(\ref{Eq_6214_Dfg}))
at the first step,
$1 \in F_1$ and~(\ref{Eq_6214_AppId}) at the third last step,
and~(\ref{6214_Ep0_e8}) at the second last step,
to estimate
\begin{align*}
\| h - g \|
& < \frac{\ep}{8}
   + \left\|
   1 - (1 - g_0) \left( 1 - \frac{1}{n} \sum_{k = 1}^{n - 1} g_k \right)
    - h
       \right\|
\\
& = \frac{\ep}{8}
    + \left\| g_0 + \frac{1}{n} \sum_{k = 1}^{n - 1} g_k
            - \frac{1}{n} \sum_{k = 1}^{n - 1} g_0 g_k
            - \frac{1}{n} \sum_{k = 0}^{n - 1} g_k \right\|
\\
& = \frac{\ep}{8}
    + \left\| \left( \frac{n - 1}{n} \right) g_0
            - \frac{1}{n} \sum_{k = 1}^{n - 1} g_0 g_k \right\|
  \leq \frac{\ep}{8}
     + \frac{1}{n} \sum_{k = 1}^{n - 1} \| g_0 - g_0 g_k \|
\\
&  < \frac{\ep}{8} + \ep_0
  \leq \frac{\ep}{8} + \frac{\ep}{8}
  = \frac{\ep}{4}.
\end{align*}
The claim is proved.

We next claim that $\| u h u^* - h \| \leq \frac{\ep}{2}$
for $u \in F$.
Set $g_{- 1} = 0$ and define
\[
e = \frac{1}{n} \sum_{k = 0}^{n - 1}
          u g_k^{1/2} u^* g_{k - 1} u g_k^{1/2} u^*
\andeqn
f = \frac{1}{n} \sum_{k = 0}^{n - 1}
          g_{k + 1}^{1/2} u g_k u^* g_{k + 1}^{1/2}.
\]
We will apply Lemma~\ref{L-2720Narrow} with
\[
a = \frac{1}{n} \sum_{k = 0}^{n - 2} g_k,
\,\,\,\,\,\,
b = \frac{1}{n} \sum_{k = 0}^{n} g_k,
\,\,\,\,\,\,
d = h,
\,\,\,\,\,\,
r = u h u^*,
\,\,\,\,\,\,
x = e,
\,\,\,\,\,\,
y = f,
\]
and with $\frac{2}{n}$ in place of~$\ep$
and $\frac{\ep}{4}$ in place of~$\rh$.
We verify its hypotheses.
The relation
\[
\frac{1}{n} \sum_{k = 0}^{n - 2} g_k
 \leq h
 \leq \frac{1}{n} \sum_{k = 0}^{n} g_k
\]
is clear.
We also have
\[
\left\| \frac{1}{n} \sum_{k = 0}^{n} g_k
            - \frac{1}{n} \sum_{k = 0}^{n - 2} g_k \right\|
 \leq \frac{1}{n} \big( \| g_{n - 1} \| + \| g_n \| \big)
 \leq \frac{2}{n}.
\]
The relation $f \leq \frac{1}{n} \sum_{k = 0}^{n} g_k$
follows from $g_0 \geq 0$
and
\[
g_{k + 1}^{1/2} u g_k u^* g_{k + 1}^{1/2}
  \leq \| u g_k u^* \| g_{k + 1}
  \leq g_{k + 1}.
\]
The proof of the inequality $e \leq u h u^*$ is similar.

We now check that
\begin{equation}\label{Eq:2724eDiff}
\left\| e - \frac{1}{n} \sum_{k = 0}^{n - 2} g_k \right\|
 \leq \frac{\ep}{4}.
\end{equation}
For $k = 1, 2, \ldots, n - 1$,
use (\ref{Eq_6214_AppId}) and~(\ref{6214_Ep0_RTAI})
at the second step to get
\[
\big\| u g_k^{1/2} u^* g_{k - 1} - g_{k - 1} \big\|
  = \big\| g_k^{1/2} u^* g_{k - 1} u - u^* g_{k - 1} u \big\|
  < \frac{\ep}{8}.
\]
Taking adjoints gives also
\[
\big\| g_{k - 1} u g_k^{1/2} u^* - g_{k - 1} \big\|
  < \frac{\ep}{8}.
\]
So
\[
\big\| u g_k^{1/2} u^* g_{k - 1} u g_k^{1/2} u^* - g_{k - 1} \big\|
   < \frac{\ep}{4}.
\]
Since $g_{- 1} = 0$,
we therefore have
\begin{align*}
\left\| e - \frac{1}{n} \sum_{k = 0}^{n - 2} g_k \right\|
& = \left\| \frac{1}{n} \sum_{k = 1}^{n - 1}
          u g_k^{1/2} u^* g_{k - 1} u g_k^{1/2} u^*
       - \frac{1}{n} \sum_{k = 1}^{n - 1} g_{k - 1} \right\|
\\
& \leq \frac{1}{n} \sum_{k = 1}^{n - 1}
     \big\| u g_k^{1/2} u^* g_{k - 1} u g_k^{1/2} u^* - g_{k - 1} \big\|
   < \frac{\ep}{4}.
\end{align*}
This is~(\ref{Eq:2724eDiff}).

Similar reasoning to that of the previous paragraph gives
\[
\big\| g_{k + 1}^{1/2} u g_k u^* g_{k + 1}^{1/2} - u g_k u^* \big\|
  < \frac{\ep}{4}
\]
for $k = 1, 2, \ldots, n$,
so that
\[
\| f - u h u^* \|
 \leq \frac{1}{n} \sum_{k = 0}^{n - 1}
  \big\| g_{k + 1}^{1/2} u g_k u^* g_{k + 1}^{1/2} - u g_k u^* \big\|
 < \frac{\ep}{4}.
\]
This completes the verification of the hypotheses of
Lemma~\ref{L-2720Narrow}.
Applying this lemma,
and using~(\ref{Eq_6213_1ne8}) at the second step,
we get
\[
\| u h u^* - h \| \leq \frac{2}{n} + \frac{\ep}{4} < \frac{\ep}{2}.
\]

Combining this estimate with~(\ref{Eq:2724gmh}),
for $u \in F$ we now have
\[
\| [g, u] \|
 = \| g - u g u^* \|
 \leq \| h - u h u^* \| + 2 \| g - h \|
 < \frac{\ep}{2} + 2 \left( \frac{\ep}{4} \right)
 = \ep.
\]
This completes the proof of condition~(\ref{D-2717RqMod-Cut5}),
and the proof of the theorem.
\end{proof}


\section{Replacing the matrix decomposition}\label{Sec:MatDec}

One frequently used technique in C*-algebras with many projections
is decomposing the identity into orthogonal projections
in order to create
a matrix decomposition of some element of interest.
Lemma~\ref{bStr1} below creates a decomposition of the identity
into three positive elements
$z_1$, $z_2$, and $z_3$ (which are not mutually orthogonal).
These elements are used to simulate a $3 \times 3$ matrix decomposition
of an element of the C*-algebra.
It is helpful to think of the following picture.

\includegraphics[height=40mm]{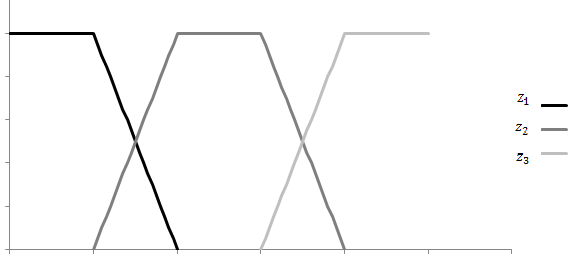}

Lemma~\ref{bcStr1} is a technical lemma
using two related decompositions
of the identity into three positive elements
and simulating a $5 \times 5$ matrix
decomposition
of an element of the C*-algebra.
Lemma~\ref{bcStr1} is used to prove Lemma~\ref{bStr1}.
The following is the motivational picture.

\includegraphics[height=40mm]{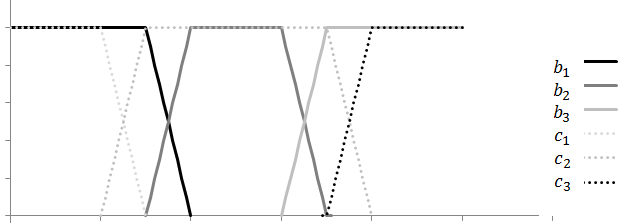}

The next lemma is used repeatedly,
but often implicitly, in the proof of Lemma~\ref{bcStr1}.

\begin{lem}\label{OrthHer}
If $y$ and $z$ are orthogonal positive elements
of a \ca~$A$ and $w \in {\overline{A y} }$
and $x \in {\overline{z A}}$,
then $w x = 0$.
\end{lem}

\begin{proof}
The statement is true if $w \in A y$ and $x \in z A$.
Take limits.
\end{proof}

\begin{lem}\label{bcStr1}
Let $A$ be a finite unital \ca{} and let $a \in A$.
Let $B \subset A$ be a unital subalgebra with $\tsr (B) = 1$.
Suppose there exist positive elements
$b_1, b_2, b_3, c_1, c_2, c_3 \in B$ such that:
\begin{enumerate}
\item\label{bcStr1_1}
$b_1 + b_2 + b_3 = 1$.
\item\label{bcStr1_2}
$c_1 + c_2 + c_3 = 1$.
\item\label{bcStr1_3}
$C^* (b_1, b_2, b_3, c_1, c_2, c_3)$ is commutative.
\item\label{bcStr1_4}
$b_1 c_1 = c_1$.
\item\label{bcStr1_5}
$b_3 c_3 = c_3$.
\item\label{bcStr1_6}
$b_2 c_2 = b_2$.
\item\label{bcStr1_7}
$c_1 b_2 = c_3 b_2 = 0$.
\item\label{bcStr1_8}
$b_1 b_3 = 0$.
\item\label{bcStr1_9}
$c_2 a c_2 \in { \overline{c_2 B c_2 }}$.
\item\label{bcStr1_10}
$b_1 a = a b_3 = 0$.
\setcounter{TmpEnumi}{\value{enumi}}
\end{enumerate}
Then for all $\ep > 0$ there exists an invertible element $y \in A$
such that $\| a - y \| < \ep$.
\end{lem}

\begin{proof}
The relations in the hypotheses imply:
\begin{equation}\label{C1B3}
c_1 b_3 = c_1 b_1 b_3 = 0.
\end{equation}
\begin{equation}\label{C1C3}
 c_1 c_3 = c_1 b_3 c_3 = 0.
\end{equation}
\begin{equation}\label{B1C3}
 b_1 c_3 = b_1 b_3 c_3 = 0.
\end{equation}
{}From~(\ref{bcStr1_1}) we get
\begin{equation}\label{Eq_4X04_Star}
1 = c_1 + (b_1 - c_1) + b_2 + (b_3 - c_3 ) + c_3.
\end{equation}
We wish to use this decomposition of the identity to decompose~$a$.
Therefore,
make the following definitions:
\begin{equation}\label{E_4920_a31}
a_{3, 1} = b_2 a c_1.
\end{equation}
\begin{equation}\label{E_4920_a32}
a_{3, 2} = b_2 a (b_1 - c_1).
\end{equation}
\begin{equation}\label{E_4920_a33}
a_{3, 3} = b_2 a b_2.
\end{equation}
\begin{equation}\label{E_4920_a41}
a_{4, 1} = (b_3 - c_3) a c_1.
\end{equation}
\begin{equation}\label{E_4920_a42}
a_{4, 2} = (b_3 - c_3 ) a (b_1 - c_1).
\end{equation}
\begin{equation}\label{E_4920_a43}
a_{4, 3} = (b_3 - c_3) a b_2.
\end{equation}
\begin{equation}\label{E_4920_a51}
a_{5, 1} = c_3 a c_1.
\end{equation}
\begin{equation}\label{E_4920_a52}
a_{5, 2} = c_3 a (b_1 - c_1).
\end{equation}
\begin{equation}\label{E_4920_a53}
a_{5, 3} = c_3 a b_2.
\end{equation}
Then
\begin{equation}\label{Eq_4925_aSum}
\sum_{j = 3}^5 \sum_{k = 1}^3 a_{j, k} = a
\end{equation}
by (\ref{Eq_4X04_Star})
and the hypotheses (\ref{bcStr1_10}),
(\ref{bcStr1_4}),
and~(\ref{bcStr1_5}).

We have,
using (\ref{bcStr1_4}) at the first step,
(\ref{bcStr1_2}) at the second step,
and~(\ref{B1C3}) at the third step,
\[
b_1 - c_1
 = b_1 (1 - c_1)
 = b_1 (c_2 + c_3)
 = b_1 c_2.
\]
Similarly $b_3 - c_3 = b_3 c_2$.
Also using~(\ref{bcStr1_6}),
it follows that
\[
a_{3, 2}, \, a_{3, 3}, \, a_{4, 2}, \, a_{4, 3}
 \in {\overline{c_2 B c_2}};
\]
in fact,
$a_{4, 3} \in {\overline{(b_3 - c_3) B b_2}}$.
Since
$b_2 c_2 = b_2$, we have
$b_2 a b_2 \in {\overline{b_2 B b_2}} \subset {\overline{c_2 B c_2}}$.
Furthermore,
since $B$ has stable rank one,
${\overline{b_2 B b_2}}$ and ${\overline{c_2 B c_2}}$
also have stable rank one.

Set
\[
\delta = \min \left(
  \frac{\ep}{6 \| a_{3, 1} + a_{4, 1} + a_{5, 1} \| }, \, \,
  \sqrt{ \frac{\ep}{6} }, \, \,
  \frac{\ep}{3} \right).
\]
Since $a_{3, 3} \in {\overline{ b_2 B b_2 }}$,
there is an invertible element
$t_0 \in {\overline{ b_2 B b_2 }} + \C \cdot 1_A$ with
\begin{equation}\label{t0a33}
\| t_0 - a_{3, 3} \| < \delta.
\end{equation}
Then there are $t_1, t_2 \in {\overline{ b_2 B b_2 }}$
and $\lambda_1 \in \C$
such that $t_0 = t_1 + \lambda_1 1_A$
and $t_0^{- 1} = t_2 + \lambda_1^{- 1} 1_A$.

We claim that
\begin{equation}\label{sqrdIs0}
\big( (a_{3, 1} + a_{4, 1} + a_{5, 1}) t_0^{- 1} \big)^2 = 0.
\end{equation}
To prove the claim, we first observe that,
using the hypothesis~(\ref{bcStr1_7}) and Lemma~\ref{OrthHer}
at the second step, we get
\[
( a_{3, 1} + a_{4, 1} + a_{5, 1}) t_0
 = (a_{3, 1} + a_{4, 1} + a_{5, 1}) (t_1 + \lambda_1 )
 = (a_{3, 1} + a_{4, 1} + a_{5, 1}) \lambda_1.
\]
Therefore
\[
a_{3, 1} + a_{4, 1} + a_{5, 1}
 = (a_{3, 1} + a_{4, 1} + a_{5, 1}) t_0 t_0^{- 1}
 = (a_{3, 1} + a_{4, 1} + a_{5, 1}) \lambda_1 t_0^{- 1}.
\]
Using (\ref{C1B3}), (\ref{C1C3}), and the hypothesis~(\ref{bcStr1_7})
at the second step,
this implies
\[
\big( (a_{3, 1} + a_{4, 1} + a_{5, 1}) t_0^{- 1} \big)^2
 = (a_{3, 1} + a_{4, 1} + a_{5, 1})^2 \lambda_1^{- 2}
 = 0.
\]
The claim is proved.

Now we compute
\begin{align}\label{1stInv}
& \left(a_{3, 1} + a_{4, 1} + a_{5, 1} + t_0 \right)
 t_0^{- 1} \left[ 1 - \left(a_{3, 1} + a_{4, 1} + a_{5, 1}\right)
 t_0^{- 1} \right]
\\
& \alignInd = \left(a_{3, 1} + a_{4, 1} + a_{5, 1}\right) t_0^{- 1}
 - \left(\left(a_{3, 1} + a_{4, 1} + a_{5, 1}\right) t_0^{- 1} \right)^2
\notag
\\
& \alignInd \alignInd
 + 1 - \left(a_{3, 1} + a_{4, 1} + a_{5, 1} \right) t_0^{- 1}
\notag
\\
& \alignInd = 1.
\notag
\end{align}
Because $A$ is finite,
every one sided invertible element is invertible.
So $a_{3, 1} + a_{4, 1} + a_{5, 1} + t_0$ is invertible,
with
\begin{equation}\label{Eq_I1}
\big( a_{3, 1} + a_{4, 1} + a_{5, 1} + t_0 \big)^{-1}
= t_0^{- 1}
  \left[ 1 - \left( a_{3, 1} + a_{4, 1} + a_{5, 1} \right)
     t_0^{- 1} \right].
\end{equation}

Using our expression for $t_0^{- 1}$ in the first step
and using Lemma~\ref{OrthHer}
twice in the second step,
once with $y = c_1$ and $z = b_2$
and once with $y = c_1$ and $z = b_3$ (using~(\ref{C1B3})),
we compute
\begin{equation}\label{Eq_4924Star}
\left( a_{3, 1} + a_{4, 1} + a_{5, 1} \right)
      t_0^{- 1} b_3
 = \left( a_{3, 1} + a_{4, 1} + a_{5, 1} \right)
     \left(t_2 + \lambda_1^{- 1}\right) b_3
 = 0.
\end{equation}
Similarly,
\begin{equation}\label{Eq_4924StarStar}
\left( a_{3, 1} + a_{4, 1} + a_{5, 1} \right) t_0^{- 1} b_2
 = 0.
\end{equation}
{}From (\ref{Eq_4924Star}) and the hypothesis~(\ref{bcStr1_5}),
we get
\begin{equation}\label{Eq_4924ThreeStar}
(a_{3, 1} + a_{4, 1} + a_{5, 1}) t_0^{- 1} c_3 = 0
\andeqn
(a_{3, 1} + a_{4, 1} + a_{5, 1}) t_0^{- 1} (b_3 - c_3) = 0.
\end{equation}

Next we multiply,
using (\ref{Eq_I1}) at the first step,
and at the last step using $t_2 \in {\overline{b_2 B b_2}}$
together with the hypothesis~(\ref{bcStr1_7})
on the terms $t_2 a_{5, 2}$ and $t_2 a_{5, 3}$,
and using (\ref{Eq_4924StarStar}) and (\ref{Eq_4924ThreeStar})
on the terms starting with $t_0^{- 1}$:
\begin{align}
& t_0^{- 1} \left[ 1 - \left( a_{3, 1} + a_{4, 1} + a_{5, 1} \right)
 t_0^{- 1} \right]
\label{InverseXAdjustment1}
\\
& \alignInd \alignInd
  \cdot \big( a_{3, 1} + a_{3, 2} + t_0 + a_{4, 1}
  + a_{4, 2} + a_{4, 3} + a_{5, 1} + a_{5, 2} + a_{5, 3} \big)
\notag
\\
& \alignInd
 = 1 + t_0^{- 1}
    \left[ 1 - \left( a_{3, 1} + a_{4, 1} + a_{5, 1} \right)
       t_0^{- 1} \right]
\notag
\\
& \alignInd \alignInd
 \cdot \big( a_{3, 2} + a_{4, 2} + a_{4, 3} + a_{5, 2} + a_{5, 3} \big)
\notag
\\
& \alignInd
 = 1 + \left[ t_2 + \lambda_1^{- 1}
    - t_0^{- 1} \left( a_{3, 1} + a_{4, 1} + a_{5, 1} \right)
       t_0^{- 1} \right]
\notag
\\
& \alignInd  \alignInd
   \cdot
   \big( a_{3, 2} + a_{4, 2} + a_{4, 3} + a_{5, 2} + a_{5, 3} \big)
\notag
\\
& \alignInd
 = 1 + t_2 a_{3, 2} + t_2 a_{4, 2} + t_2 a_{4, 3}
       + t_2 a_{5, 2} + t_2 a_{5, 3}
\notag
\\
& \alignInd \alignInd
 + \lambda_1^{- 1} a_{3, 2} + \lambda_1^{- 1} a_{4, 2}
 + \lambda_1^{- 1} a_{4, 3} + \lambda_1^{- 1} a_{5, 2}
 + \lambda_1^{- 1} a_{5, 3}
\notag
\\
& \alignInd \alignInd
 - t_0^{- 1} (a_{3, 1} + a_{4, 1} + a_{5, 1}) t_0^{- 1}  a_{3, 2}
\notag
\\
& \alignInd \alignInd
 - t_0^{- 1} (a_{3, 1} + a_{4, 1} + a_{5, 1}) t_0^{- 1}
       (a_{4, 2} + a_{4, 3} )
\notag
\\
& \alignInd \alignInd
  - t_0^{- 1} (a_{3, 1} + a_{4, 1} + a_{5, 1}) t_0^{- 1}
     (a_{5, 2} + a_{5, 3})
\notag
\\
& \alignInd
 = 1 + t_2 a_{3, 2} + t_2 a_{4, 2} + t_2 a_{4, 3}
  + \lambda_1^{- 1} a_{3, 2}
\notag
\\
& \alignInd \alignInd
  + \lambda_1^{- 1} a_{4, 2} + \lambda_1^{- 1} a_{4, 3}
  + \lambda_1^{- 1} a_{5, 2} + \lambda_1^{- 1} a_{5, 3}.
\notag
\end{align}

Set $t_3 = t_2 a_{4, 3} + 1$.
Since $t_2 \in {\overline {b_2 B b_2 }}$ and
$a_{4, 3} \in {\overline {(b_3 - c_3) B b_2 }}$,
we have $t_3 \in {\overline {b_2 B b_2 }} + \C \cdot 1_A$.
Thus there is an invertible element
$t_4 \in {\overline {b_2 B b_2 }} + \C \cdot 1_A$ with
\begin{equation}\label{t4tot3}
\| t_4 - t_3 \| < \delta.
\end{equation}
So there are $t_5, t_6 \in {\overline {b_2 B b_2 }}$
and $\lambda_5 \in \C$
such that $t_4 = t_5 + \lambda_5 1_A$
and $t_4^{- 1} = t_6 + \ld_5^{- 1} 1_A$.

An argument similar to that for Equation~(\ref{sqrdIs0}),
using (\ref{C1C3}), (\ref{B1C3}), and the hypothesis~(\ref{bcStr1_7})
gives
$\big( t_4^{- 1} ( a_{5, 2} + a_{5, 3} ) \big)^2 = 0$.
Define
\[
w = \ld_1^{- 1} a_{5, 2} + \ld_1^{- 1} a_{5, 3}.
\]
Thus
\begin{equation}\label{Eq_4724_N1}
\big( t_4^{- 1} w \big)^2 = 0.
\end{equation}
A computation similar to the one for Equation~(\ref{1stInv})
shows that $w + t_4$ is invertible,
with
\begin{equation}\label{inverseto2ndadj}
( w + t_4 )^{- 1}
 = \big( 1 - t_4^{- 1} w \big) t_4^{- 1}.
\end{equation}
Applying Lemma~\ref{OrthHer},
and using (\ref{B1C3}),
the hypothesis~(\ref{bcStr1_7}),
and the relations
\[
t_6 \in {\overline {b_2 B b_2 }},
\,\,\,\,\,\,
b_1 \in {\overline{ b_1 B b_1}},
\andeqn
a_{5, 2} + a_{5, 3} \in {\overline{ c_3 A}},
\]
we get
\begin{equation}\label{Zeros}
b_1 t_4^{- 1} w
= b_1 t_6 w + b_1 \ld_5^{- 1} w
 = 0.
\end{equation}
Similarly,
\begin{equation}\label{Eq_4724_N3}
b_2 t_4^{- 1} w = 0.
\end{equation}
Hypothesis~(\ref{bcStr1_4})
and the relation~(\ref{Zeros}) also imply that
\begin{equation}\label{Eq_4724_N2}
c_1 t_4^{- 1} w = 0
\andeqn
(b_1 - c_1) t_4^{- 1} w = 0.
\end{equation}

Define
\begin{align*}
t_7
& = 1 + t_2 a_{3, 2} t_6 + t_2 a_{4, 2} t_6
     + \ld_1^{- 1} a_{3, 2} t_6 + \ld_1^{- 1} a_{4, 2} t_6
     + \ld_1^{- 1} a_{4, 3} t_6
\\
& \alignInd + \ld_5^{- 1} t_2 a_{3, 2} + \ld_5^{- 1} t_2 a_{4, 2}
     + \ld_1^{- 1} \ld_5^{- 1} a_{3, 2}
     + \ld_1^{- 1} \ld_5^{- 1} a_{4, 2}
     + \ld_1^{- 1} \ld_5^{- 1} a_{4, 3}.
\end{align*}
One checks that
$t_7 \in {\overline {c_2 B c_2}} + \C \cdot 1_A$,
which has stable rank one by hypothesis.
Thus there exists an invertible element
$t_8 \in {\overline {c_2 B c_2 }} + \C \cdot 1_A$ such that
\begin{equation}\label{t7tot8}
\| t_8 - t_7 \| < \frac{\delta}{1 + \| w \|}.
\end{equation}%
Set
\[
y = \left( a_{3, 1} + a_{4, 1} + a_{5, 1} + t_0 \right) t_8
    ( w + t_4 ).
\]
Then $y$ is invertible
by (\ref{Eq_I1}) and~(\ref{inverseto2ndadj}).

We claim that $\| a - y \| < \ep$.
To prove the claim,
we begin with the following computation.
In the first step we multiply out,
cancel the terms $\pm w t_4^{-1}$,
and use~(\ref{Eq_4724_N1})
to see that $w t_4^{-1} w t_4^{-1} = 0$.
The second step uses
(\ref{Eq_4724_N3}) and~(\ref{Eq_4724_N2})
to see that the terms involving $t_4^{-1} w t_4^{-1}$ vanish.
The third step is obtained by substituting
$t_4^{-1} = t_6 + \ld_5^{-1} 1_A$
and comparing with the definition of~$t_7$.
We have:
\begin{align}
& \left( t_2 a_{3, 2} + t_2 a_{4, 2} + t_4 + \ld_1^{- 1} a_{3, 2}
   + \ld_1^{- 1} a_{4, 2}
   + \ld_1^{- 1} a_{4, 3} + w \right)
 \cdot \left[ 1 - t_4^{- 1} w \right] t_4^{- 1}
     \label{BigMult}
\\
& \alignInd
  = 1 + t_2 a_{3, 2} t_4^{- 1} + t_2 a_{4, 2} t_4^{- 1} + \ld_1^{- 1} a_{3, 2} t_4^{- 1}
 + \ld_1^{- 1} a_{4, 2} t_4^{- 1} + \ld_1^{- 1} a_{4, 3} t_4^{- 1}
\notag
\\
& \alignInd \alignInd
 - t_2 a_{3, 2} t_4^{- 1} w t_4^{- 1}
   - t_2 a_{4, 2} t_4^{- 1} w t_4^{- 1}
\notag
\\
& \alignInd \alignInd
   - \ld_1^{- 1} a_{3, 2} t_4^{- 1} w t_4^{- 1}
  - \ld_1^{- 1} a_{4, 2} t_4^{- 1} w t_4^{- 1}
  - \ld_1^{- 1} a_{4, 3} t_4^{- 1} w t_4^{- 1}
\notag
\\
& \alignInd
 = 1 + t_2 a_{3, 2} t_4^{- 1} + t_2 a_{4, 2} t_4^{- 1}
     + \ld_1^{- 1} a_{3, 2} t_4^{- 1} + \ld_1^{- 1} a_{4, 2} t_4^{- 1}
     + \ld_1^{- 1} a_{4, 3} t_4^{- 1}
\notag
\\
& \alignInd
 = t_7.
\notag
\end{align}
Next,
using (\ref{inverseto2ndadj}) at the first step,
(\ref{BigMult}) at the second step,
and the choices of $t_4$ and $t_8$ at the third step,
we get
\begin{align}\label{Partial}
& \big\| 1 + t_2 a_{3, 2} + t_2 a_{4, 2} + t_2 a_{4, 3}
    + \ld_1^{- 1} a_{3, 2} + \ld_1^{- 1} a_{4, 2}
    + \ld_1^{- 1} a_{4, 3}
    + w - t_8 (w + t_4 ) \big\|
\\
\notag
& \alignInd
\leq \| 1 + t_2 a_{4, 3} - t_4 \|
\\
\notag
& \alignInd \alignInd
  + \big\| \big( t_2 a_{3, 2} + t_2 a_{4, 2} + t_4
    + \ld_1^{- 1} a_{3, 2} + \ld_1^{- 1} a_{4, 2}
    + \ld_1^{- 1} a_{4, 3} + w \big)
\\
\notag
& \alignInd \alignInd \alignInd \alignInd
  \cdot \big(1 - t_4^{- 1} w \big) t_4^{- 1} (w + t_4)
\\
\notag
& \alignInd \alignInd \alignInd
     - t_8 \big( w + t_4 \big) \big\|
\\
\notag
& \alignInd
 \leq \| 1 + t_2 a_{4, 3} - t_4 \|
        + \big\| t_7 - t_8 \big\| \cdot \| w + t_4 \|
\\
\notag
& \alignInd
\leq \dt
  + \left( \frac{\delta}{\| w + t_4 \| + 1} \right) \| w + t_4 \|
< 2 \dt.
\end{align}
Then we compute as follows,
using (\ref{Eq_4925_aSum}) and the definitions of $y$ and $w$
at the first step,
(\ref{1stInv}) at the second step,
(\ref{InverseXAdjustment1}) and the definition of $w$ at the third step,
and~(\ref{Partial}) at the fourth step,
\begin{align*}
\| a - y \|
&
= \big\| a_{3, 1} + a_{3, 2} + a_{3, 3} + a_{4, 1} + a_{4, 2}
    + a_{4, 3} + a_{5, 1} + \ld_1 w
\\
& \alignInd
 - (a_{3, 1} + a_{4, 1} + a_{5, 1} + t_0) t_8 (w + t_4 ) \big\|
\\
&
\leq \| a_{3, 3} - t_0 \|
\\
& \alignInd
 + \big\| \big( a_{3, 1} + a_{4, 1} + a_{5, 1} + t_0 \big) t_0^{- 1}
   \big[ 1 - (a_{3, 1} + a_{4, 1} + a_{5, 1}) t_0^{- 1} \big]
\\
& \alignInd \alignInd \alignInd
 \cdot \big( a_{3, 1} + a_{3, 2} + t_0 + a_{4, 1} + a_{4, 2}
   + a_{4, 3} + a_{5, 1} + \ld_1 w \big)
\\
& \alignInd \alignInd
 - ( a_{3, 1} + a_{4, 1} + a_{5, 1} + t_0) t_8
    ( w + t_4 ) \big\|
\\
&
\leq \delta + \| a_{3, 1} + a_{4, 1} + a_{5, 1} + t_0 \|
\\
& \alignInd
  \cdot \| 1 + t_2 a_{3, 2} + t_2 a_{4, 2} + t_2 a_{4, 3}
    + \ld_1^{- 1} a_{3, 2} + \ld_1^{- 1} a_{4, 2}
    + \ld_1^{- 1} a_{4, 3}
\\
& \alignInd \alignInd
    + w - t_8 (w + t_4 ) \|
\\
&
 \leq \delta + 2 \dt \| a_{3, 1} + a_{4, 1} + a_{5, 1} + t_0 \|
\\
&
 \leq \delta
  + 2 \delta \| a_{3, 1} + a_{4, 1} + a_{5, 1} + a_{3, 3} \|
  + 2 \delta \| a_{3, 3} - t_0 \|
 < \frac{\ep}{3} + \frac{2 \ep}{6} + 2 \delta^2
 \leq \ep.
\end{align*}
This proves the claim and completes the proof.
\end{proof} 

\begin{lem}\label{StillIn}
Let $A$ be a \ca{} and let $B \subset A$ be a subalgebra.
Suppose $a \in A$ and $b \in B_{+}$ satisfy $b a b \in B$.
Let $f \colon \spec (b) \to [0, \infty)$
be a continuous function vanishing at zero.
Then $f (b) a f (b) \in {\overline{f (b) B f (b)}}$.
\end{lem}

\begin{proof}
We have $b^{1/n} a b^{1/n} \in B$
by a standard polynomial approximation argument.
Also,
for all $k \in \N$ we have $\lim_{n \to \infty} b^k b^{1 / n} = b^k$.
Therefore polynomial approximation gives
$\lim_{n \to \infty} f (b) b^{1 / n} = f (b)$.
Similarly $\lim_{n \to \infty} b^{1 / n} f (b) = f (b)$.
The result follows.
\end{proof}

\begin{lem}\label{bStr1}
Let $A$ be a finite unital \ca{}
and let $a \in A$.
Let $B \subset A$ be a unital subalgebra with $\tsr (B) = 1$.
Suppose there exist
$z_{1}, z_{2}, z_{3} \in B$
such that:
\begin{enumerate}
\item\label{26640}
$z_{1}, z_{2}, z_{3} \geq 0$.
\item\label{26641}
$z_{1} + z_{2} + z_{3} = 1$.
\item\label{26642}
$z_{1} z_{3} = 0$.
\item\label{26643}
$z_{1} a = 0$.
\item\label{26644}
$a z_{3} = 0$
\item\label{26645}
$z_2 a z_2 \in B$.
\end{enumerate}
Then for all $\ep > 0$ there exists
an invertible element $y \in A$
such that $\| y - a \| < \ep$.
\end{lem}

\begin{proof}
We verify the hypotheses of Lemma~\ref{bcStr1}.
Let $f, h \colon [0, 1] \to [0, 1]$
be defined by the formulas
\[
f (\ld) = \begin{cases}
     0         & \ld \in \big[ 0, \tfrac{1}{2} \big]
      \\
     2 \ld - 1 & \ld \in \big[ \tfrac{1}{2}, 1 \big]
\end{cases}
\,\,\,\,\,\,
\andeqn
\,\,\,\,\,\,
h (\ld) = \begin{cases}
    2 \ld & \ld \in \big[ 0, \tfrac{1}{2} \big]
     \\
     1    & \ld \in \big[ \tfrac{1}{2}, 1 \big].
\end{cases}
\]
Then $f h = f$.

Define positive elements
$b_{1}, b_{2}, b_{3}, c_{1}, c_{2}, c_{3} \in A$
by
\[
b_{1} = h (z_{1}),
\,\,\,\,\,\,
b_{2} = f (z_{2}),
\,\,\,\,\,\,
b_{3} = h (z_{3}),
\,\,\,\,\,\,
c_{1} = f (z_{1}),
\,\,\,\,\,\,
c_{2} = h (z_{2}),
\,\,\,\,\,\,
c_{3} = f (z_{3}).
\]
Set $D = C^* (z_1, z_2, z_3)$.
Since $D = C^* (1, z_1, z_3)$ and $z_1 z_3 = 0$,
we see that $D$ is commutative.
So there is a compact metric space~$X$
and an isomorphism $D \cong C (X)$,
which we treat as an identification.
It also follows that
$C^* (b_{1}, b_{2}, b_{3}, c_{1}, c_{2}, c_{3}) \subset D$
is commutative.
We have
$c_2 a c_2 \in {\overline{c_2 B c_2}}$
by Lemma~\ref{StillIn}.

Identifying $D = C (X)$ as above,
for $j = 1, 2, 3$ set
\[
U_j = \big\{ x \in X \colon z_j (x) > 0 \big\}
\andeqn
V_j = \big\{ x \in X \colon z_j (x) > \tfrac{1}{2} \big\}.
\]
{}From the definitions of $f$ and $h$,
we see that if $\ld_1, \ld_2 \in [0, 1]$ satisfy $\ld_1 + \ld_2 = 1$,
then $f (\ld_1) + h (\ld_2) = 1$.
Since $U_1 \cap U_3 = \varnothing$
and $z_1 + z_2 + z_3 = 1$,
it is now easily checked that
\[
b_1 + b_2 + b_3 = 1,
\,\,\,\,\,\,
c_1 + c_2 + c_3 = 1,
\andeqn
b_1 b_3 = 0.
\]
The relation $f h = f$
implies
\[
b_1 c_1 = c_1,
\,\,\,\,\,\,
b_2 c_2 = b_2,
\andeqn
b_3 c_3 = c_3.
\]

For $x \in X$,
we have $c_1 (x) \neq 0$ \ifo{} $x \in V_1$
and $b_2 (x) \neq 0$ \ifo{} $x \in V_2$.
Clearly $V_1 \cap V_2 = \varnothing$,
so $c_1 b_2 = 0$.
Similarly $c_3 b_2 = 0$.
We have $b_1 a = 0$ by Lemma~\ref{OrthHer}
and similarly $a b_3 = 0$.
This completes the verification of the hypotheses of Lemma~\ref{bcStr1}.
\end{proof} 


\section{Centrally large subalgebras and stable rank
 one}\label{Sec:RqTsr1}

In this section we will prove the main theorem
(Theorem~\ref{tsr1GoesUpExtraLarge}),
which says that a \ca{} with a
centrally large subalgebra
which has stable rank one
must itself have stable rank one.
If the subalgebra
also has real rank zero,
we further show that
the containing algebra does as well.

The proof of Theorem~\ref{tsr1GoesUpExtraLarge}
proceeds via two technical lemmas.
The first is a version of the definition of a centrally large subalgebra
in which the element $g$ is replaced by a tower of elements.

\begin{lem}\label{L_9430_LgTower}
Let $A$ be an in\fd{} simple separable unital \ca,
and let $B \subset A$ be a centrally large subalgebra.
Then for all $m, N \in \N$,
$a_1, a_2, \ldots, a_m \in A$,
$\ep > 0$, $x \in A_{+}$ with $\| x \| = 1$,
and $y \in B_{+} \SM \{ 0 \}$,
there are $c_1, c_2, \ldots, c_m \in A$
and $g_0, g_1, \ldots, g_N \in B$
such that:
\begin{enumerate}
\item\label{9430_LgTower-Cut1}
$0 \leq g_n \leq 1$ for $n = 0, 1, \ldots, N$
and $g_{n - 1} g_n = g_n$ for $n = 1, 2, \ldots, N$.
\item\label{9430_LgTower-Cut2}
For $j = 1, 2, \ldots, m$ we have
$\| c_j - a_j \| < \ep$.
\item\label{9430_LgTower-Cut3}
For $j = 1, 2, \ldots, m$
and $n = 0, 1, \ldots, N$, we have
$(1 - g_n) c_j \in B$.
\item\label{9430_LgTower-Cut4}
For $n = 0, 1, \ldots, N$, we have
$g_n \subeq_B y$ and $g_n \subeq_A x$.
\item\label{9430_LgTower-Cut5a}
For $n = 0, 1, \ldots, N$, we have
$\| (1 - g_n) x (1 - g_n) \| > 1 - \ep$.
\item\label{9430_LgTower-Cut5}
For $j = 1, 2, \ldots, m$
and $n = 0, 1, \ldots, N$, we have
$\| g_n a_j - a_j g_n \| < \ep$.
\end{enumerate}
\end{lem}

As we will see in the proof,
it is enough to take $n = 0$ in~(\ref{9430_LgTower-Cut4})
and~(\ref{9430_LgTower-Cut5a})
and $n = N$ in (\ref{9430_LgTower-Cut3}).

\begin{proof}[Proof of Lemma~\ref{L_9430_LgTower}]
\Wolog{} $\| a_j \| \leq 1$ for $j = 1, 2, \ldots, m$
and $\ep \leq1$.

Define \cfn{s} $k_0, k_1, \ldots, k_N \colon [0, 1] \to [0, 1]$
as follows.
Set
\[
k_0 (\ld) = \begin{cases}
   \big( 1 - \tfrac{\ep}{4} \big)^{-1} \ld
             & 0 \leq \ld \leq 1 - \frac{\ep}{4}
        \\
   1         & 1 - \frac{\ep}{4} \leq \ld \leq 1.
\end{cases}
\]
For $n = 1, 2, \ldots, N$,
set
\[
k_n (\ld) = \begin{cases}
   0         & 0 \leq \ld \leq 1 - 2^{- n - 1} \ep
        \\
   2^{n + 2} \ep^{-1} \big( \ld - 1 + 2^{- n - 1} \ep \big)
             &
    1 - 2^{- n - 1} \ep \leq \ld \leq 1 - 2^{- n - 2} \ep
        \\
   1         & 1 - 2^{- n - 2} \ep \leq \ld \leq 1.
\end{cases}
\]
Then $k_{n - 1} k_n = k_n$ for $n = 1, 2, \ldots, N$
and $| k_0 (\ld) - \ld | \leq \frac{\ep}{4}$
for all $\ld \in [0, 1]$.
For $n = 0, 1, \ldots, N$,
in Lemma~\ref{FuncCalcCom}
choose $\dt_n > 0$ for $f = k_n$ and with $\ep$ as given.
Set
\[
\dt = \min \left( \dt_0, \dt_1, \ldots, \dt_N, \tfrac{\ep}{2} \right).
\]
Apply Definition~\ref{D_6213_Large}
with $\dt$ in place of~$\ep$
and with $m, a_1, a_2, \ldots, a_m, x, y$ as given,
getting  $c_1, c_2, \ldots, c_m \in A$ and $g \in B$ as there.
For $n = 0, 1, \ldots, N$,
set $g_n = k_n (g)$.
We verify the conclusions of the lemma.

Conclusion~(\ref{9430_LgTower-Cut1}) is clear from the construction.
Conclusion~(\ref{9430_LgTower-Cut2}) is immediate.
For conclusion~(\ref{9430_LgTower-Cut3}),
observe that $1 - g_N$ is gotten from $1 - g$
by functional calculus using the function
$\ld \mapsto 1 - k_N (1 - \ld)$.
This function vanishes at~$0$,
so for $j = 1, 2, \ldots, m$
we have $(1 - g_N) c_j \in B$
by Corollary~\ref{C_4926_InSub}.
Now for $n = 0, 1, \ldots, N - 1$,
we have $(1 - g_n) c_j = (1 - g_n) (1 - g_N) c_j \in B$.
For conclusion~(\ref{9430_LgTower-Cut4}),
use Lemma~2.2(i) of~\cite{KR}
to get $g_0 \subeq g$.
For $n = 1, 2, \ldots, N$,
we then get $g_n \subeq g$ since $g_n \leq g_0$.

For conclusion~(\ref{9430_LgTower-Cut5a}),
we start with the calculation
\[
\| (1 - g_0) x (1 - g_0) \|
  \geq \| (1 - g) x (1 - g) \|
      - 2 \| g - g_0 \|
  > 1 - \frac{\ep}{2} - 2 \left( \frac{\ep}{4} \right)
  = 1 - \ep.
\]
For $n = 1, 2, \ldots, N$,
use $(1 - g_n)^2 \geq (1 - g_0)^2$
at the second step to get
\begin{align*}
\| (1 - g_n) x (1 - g_n) \|
& = \| x^{1/2} (1 - g_n)^2 x^{1/2} \|
  \geq \| x^{1/2} (1 - g_0)^2 x^{1/2} \|
  \\
& = \| (1 - g_0) x (1 - g_0) \|
  > 1 - \ep.
\end{align*}
Conclusion~(\ref{9430_LgTower-Cut5})
follows directly from the choice of~$\dt_n$.
\end{proof}

\begin{lem}\label{LemBothTimesAround}
Let $m, n, N \in \N$.
Let $A$ be
an infinite dimensional simple separable
unital \ca.
Let $B \subset A$ be \rokhLg{} in~$A$
and suppose that $\tsr (B) = 1$.
Let $x_1, x_2, \ldots, x_n \in A_{+}$
be positive elements of norm one.
Let $a_1, a_2, \ldots, a_m \in A$ be arbitrary.
Then for all $\ep > 0$ there exists $h \in B$ with $0 \leq h \leq 1$,
there exist  $r_{k, l} \in B$
for $k = 1, 2, \ldots, n$ and for $l = 0, 1, \ldots, N - 1$,
and there exist $b_1, b_2 \ldots, b_m \in A$ such that:
\begin{enumerate}
\item\label{BothTimes7}
$0 \leq r_{k, l} \leq 1$ and $\| r_{k, l} \| = 1$
for $k = 1, 2, \ldots, n$ and $l = 0, 1, \ldots, N$.
\item\label{BothTimes1}
$\| x_k r_{k, 0} - r_{k, 0}\| < \ep$ for $k = 1, 2, \ldots, n$.
\item\label{BothTimes2}
$1 - h \subeq_B r_{k, N}$ for $k = 1, 2, \ldots, n$.
\item\label{BothTimes3}
$r_{k, l} r_{k, l + 1} = r_{k, l + 1}$ for $k = 1, 2, \ldots, n$
and $l = 0, 1, \ldots, N - 2$.
\item\label{BothTimes4}
$h r_{k, 0} = r_{k, 0}$ for $k = 1, 2, \ldots, n$.
\item\label{BothTimes_New}
$r_{k, l} \in {\overline{h B h}}$
for $k = 1, 2, \ldots, n$ and $l = 0, 1, \ldots, N$.
\item\label{BothTimes5}
$\| a_j - b_j \| < \ep$ for $j = 1, 2, \ldots, m$.
\item\label{BothTimes6}
$h b_j \in {\overline{h B}}\subset B$ for $j = 1, 2, \ldots, m$.
\setcounter{TmpEnumi}{\value{enumi}}
\end{enumerate}
\end{lem}

\begin{proof}
The subalgebra $B$ is simple by Proposition~5.2 of~\cite{Ph40},
finite by Proposition~3.1 of~\cite{Rf},
and infinite dimensional by Proposition~5.5 of~\cite{Ph40}.

Set $\ep_0 = \min \big( \tfrac{\ep}{4}, \tfrac{1}{48} \big)$.
Define \cfn{s} $p_0, p_1 \colon [0, 1] \to [0, 1]$
by
\[
p_0 (\ld) = \begin{cases}
   ( 1 - 2 \ep_0 )^{-1} \ld
             & 0 \leq \ld \leq 1 - 2 \ep_0
        \\
   1         & 1 - 2 \ep_0 \leq \ld \leq 1
\end{cases}
\]
and
\[
p_1 (\ld) = \begin{cases}
   0         & 0 \leq \ld \leq 1 - 2 \ep_0
        \\
   \ep_0^{-1} ( \ld - 1 + 2 \ep_0 )
             & 1 - 2 \ep_0 \leq \ld \leq 1 - \ep_0
        \\
   1         & 1 - \ep_0 \leq \ld \leq 1.
\end{cases}
\]
Define $x_k^{(0)} = p_0 (x_k)$ and $x_k^{(1)} = p_1 (x_k)$
for $k = 1, 2, \ldots, n$.
This gives
\begin{equation}\label{Eq_4X02_xk0}
\big\| x_k - x_k^{(0)} \big\| \leq 2 \ep_0,
\,\,\,\,\,\,
x_k^{(0)} x_k^{(1)} = x_k^{(1)},
\andeqn
\big\| x_k^{(1)} \big\| = 1.
\end{equation}
We further need the functions
$f_{1 / 2^{N + 2}}, \, f_{1 / 2^{N + 1}}, \, \ldots, f_{1 / 4 }
  \colon [0, 1] \to [0, 1]$
of Definition~\ref{fSubEp}.
We name them instead,
in order,
$f_0, f_1, \ldots, f_N$,
getting
\begin{equation}\label{Eq_4X02_fSubEp}
f_l f_{l + 1} = f_{l + 1}
\end{equation}
for $l = 0, 1, \ldots, N - 1$
and
\begin{equation}\label{Eq_4X02_EqHalf}
f_N |_{ \left[ \frac{1}{2}, 1 \right] } = 1.
\end{equation}

For $k = 1, 2, \ldots, n$,
use Lemma~5.3 of~\cite{Ph40}
(with $r = 1$)
to choose $y_k^{(0)} \in B_{+}$
such that
\[
\big\| y_k^{(0)} \big\| = 1
\andeqn
\big\| \big( x_k^{(1)} \big)^{1/2} y_k^{(0)} - y_k^{(0)} \big\|
 < \frac{1}{48}.
\]
Define
\[
y_k = p_0 \big( y_k^{(0)} \big)
\andeqn
z_k = p_1 \big( y_k^{(0)} \big).
\]
Then $y_k z_k = z_k$ and $\| z_k \| = 1$.
Also,
\[
\big\| y_k - y_k^{(0)} \big\|
 \leq 2 \ep_0
 \leq \frac{1}{24},
\]
so
\begin{equation}\label{Eq_4X02_xkhalfyk}
\big\| \big( x_k^{(1)} \big)^{1/2} y_k - y_k \big\|
  \leq 2 \big\| y_k - y_k^{(0)} \big\|
       + \big\| \big( x_k^{(1)} \big)^{1/2} y_k^{(0)} - y_k^{(0)} \big\|
  < \frac{1}{12} + \frac{1}{48}
  = \frac{5}{48}.
\end{equation}

For $k = 1, 2, \ldots, n$,
Lemma~2.2 or Lemma~2.3 of~\cite{Ph40}
(depending on whether ${\overline{z_k B z_k}}$ is unital)
provides nonzero orthogonal unitarily equivalent positive elements
$c_{k, 1}, c_{k, 2} \in {\overline{z_k B z_k}}$.
We may assume that $\| c_{k, 1} \| = \| c_{k, 2} \| = 1$.
Since $B$ is finite and infinite dimensional,
we can use Lemma~2.9 of~\cite{Ph40}
to find $d_k \in B_{+} \setminus \{ 0 \}$
such that whenever $g \in B$ satisfies
$0 \leq g \leq 1$ and $g \subeq_B d_k$,
then
\begin{equation}\label{Eq_4X02_56}
\| (1 - g) z_k (1 - g) \| > \frac{5}{6}.
\end{equation}
Now use simplicity of $B$ and Lemma~2.6 of~\cite{Ph40}
to find a nonzero positive element $c \in B$
such that for $k = 1, 2, \ldots, n$ we have
\begin{equation}\label{Eq_4X02_UnifSub}
c \subeq_B d_k
\andeqn
c \subeq_B \big( c_{k, 1} - \tfrac{3}{4} \big)_{+}.
\end{equation}

By a polynomial approximation argument,
there is $\dt_0 > 0$
such that whenever $D$ is a \ca{}
and $y_1, y_2 \in D$ satisfy $0 \leq y_1, y_2 \leq 1$
and $\| y_1 - y_2 \| < 2 \dt_0$,
then $\| f_0 (y_1) - f_0 (y_2) \| < \frac{\ep}{4}$.
Set $\dt = \min \big( \ep, \dt_0, \frac{1}{48} \big)$.
Apply Lemma~\ref{L_9430_LgTower}
with $N = 3$ and $\dt$ in place of~$\ep$,
getting $g_0, g_1, g_2, g_3 \in B$
and $b_1, b_2, \ldots, b_m \in A$ such that:
\begin{enumerate}
\setcounter{enumi}{\value{TmpEnumi}}
\item\label{4X02_Cut1}
$0 \leq g_0, g_1, g_2, g_3 \leq 1$,
$g_0 g_1 = g_1$, $g_1 g_2 = g_2$, and $g_2 g_3 = g_3$.
\item\label{4X02_Cut2}
$\| b_j - a_j \| < \dt$ and $(1 - g_3) b_j \in B$
for $j = 1, 2, \ldots, m$.
\item\label{4X02_Cut3}
$g_0 \subeq_B c$.
\item\label{4X02_Cut4}
$\big\| \big[ g_1, \big( x_k^{(1)} \big)^{1/2} \big] \big\| < \dt$
and
$\big\| \big[ g_2, \big( x_k^{(1)} \big)^{1/2} \big] \big\| < \dt$
for $k = 1, 2, \ldots, n$.
\end{enumerate}
Part (\ref{BothTimes5}) of the conclusion is
immediate from the choice of $b_1, b_2, \ldots, b_m$.
Set $h = 1 - g_3$.
Then $h b_j \in B$ by~(\ref{4X02_Cut2}).
Thus $\lim_{n \to \infty} h^{1/n} h b_j = h b_j \in {\overline{h B}}$, which is part~(\ref{BothTimes6}) of the conclusion.

For $k = 1, 2, \ldots, n$,
further define
$r_k = (1 - g_2) x_k^{(1)} (1 - g_2)$
and, for $l = 0, 1, \ldots, N$,
define $r_{k, l} = f_l (r_k)$.
Clearly $0 \leq r_{k, l} \leq 1$.
{}From~(\ref{Eq_4X02_fSubEp})
we get
\[
r_{k, l} r_{k, l + 1} = r_{k, l + 1}
\]
for $l = 0, 1, \ldots, N - 1$.
This is part~(\ref{BothTimes3})
and most of part~(\ref{BothTimes7}) of the conclusion.
We next use~(\ref{Eq_4X02_xkhalfyk}) twice to
get
\begin{equation}\label{Eq_4X04_New}
\big\| \big( x_k^{(1)} \big)^{1/2} y_k \big( x_k^{(1)} \big)^{1/2}
        - y_k \big\|
  < \frac{5}{24}.
\end{equation}
Using this relation at the second step,
$z_k \leq y_k$ at the third step,
and $g_2 \leq g_0$, (\ref{4X02_Cut3}), (\ref{Eq_4X02_UnifSub}),
and the choice of~$d_k$
(which gave~(\ref{Eq_4X02_56})) at the fourth step,
we get the estimate
\begin{align*}
\| r_k \|
& \geq \big\| (1 - g_2) \big( x_k^{(1)} \big)^{1/2}
     y_k \big( x_k^{(1)} \big)^{1/2} (1 - g_2) \big\|
  > \| (1 - g_2) y_k (1 - g_2) \| - \frac{5}{24}
\\
& \geq \| (1 - g_2) z_k (1 - g_2) \| - \frac{5}{24}
  > \frac{5}{6} - \frac{5}{24}
  > \frac{1}{2}.
\end{align*}
In particular,
by~(\ref{Eq_4X02_EqHalf}),
we have $\| r_{k, N} \| = 1$.
This implies that $\| r_{k, l} \| \geq 1$
for $r_{k, l} r_{k, l + 1} = r_{k, l + 1}$,
and finishes the proof of part~(\ref{BothTimes7}) of the conclusion.

Part~(\ref{BothTimes4}) of the conclusion
follows from $h = 1 - g_3$,
$g_2 g_3 = g_3$,
and the fact that
$r_{k, 0} \in {\overline{(1 - g_2) B (1 - g_2)}}$.
Combining this with part~(\ref{BothTimes3}),
we get part~(\ref{BothTimes_New}).

Our next goal is part~(\ref{BothTimes2}) of the conclusion.
We start with
\begin{equation}\label{Eq_4X02_hSub}
1 - h = g_3 \leq g_0.
\end{equation}

We next claim that
\begin{equation}\label{Eq_4X02_g0Sb}
g_0 \subeq_B \big[ (1 - g_1) y_k (1 - g_1) - \tfrac{3}{4} \big]_{+}.
\end{equation}
We use (\ref{4X02_Cut3}) and (\ref{Eq_4X02_UnifSub}) at the first step,
Lemma~1.7 of~\cite{Ph40},
$c_{k, 1} c_{k, 2} = 0$,
unitary equivalence of $c_{k, 1}$ and $c_{k, 2}$,
and $c_{k, 1} + c_{k, 2} \leq y_k$ at the second step,
Lemma~1.8 of~\cite{Ph40}
at the third step,
and $g_0 g_1 = g_1$ and Lemma 1.10 of~\cite{Ph40}
at the fourth step,
getting
\begin{align*}
g_0 \oplus g_0
& \subeq_B \big( c_{k, 1} - \tfrac{3}{4} \big)_{+}
             + \big( c_{k, 2} - \tfrac{3}{4} \big)_{+}
\\
& \subeq_B \big( y_{k} - \tfrac{3}{4} \big)_{+}
\\
& \subeq_B \big[ (1 - g_1) y_k (1 - g_1) - \tfrac{3}{4} \big]_{+}
        \oplus g_1
\\
& \subeq_B \big[ (1 - g_1) y_k (1 - g_1) - \tfrac{3}{4} \big]_{+}
        \oplus \big( g_0 - \tfrac{1}{2} \big)_{+}.
\end{align*}
Since $\tsr (B) = 1$,
we can apply Theorem~4.3 of~\cite{RW}
to get~(\ref{Eq_4X02_g0Sb}),
as desired.

Now we claim that
\begin{equation}\label{Eq_4X02_ykSbxk1}
\big[ (1 - g_1) y_k (1 - g_1) - \tfrac{3}{4} \big]_{+}
  \subeq_A \big( r_k - \tfrac{1}{2} \big)_{+}.
\end{equation}
Set
\[
s = \big[ (1 - g_2) \big( x_k^{(1)} \big)^{1/2} (1 - g_2) \big]
       \big[ (1 - g_1) y_k (1 - g_1) \big]
       \big[ (1 - g_2) \big( x_k^{(1)} \big)^{1/2} (1 - g_2) \big].
\]
Then,
using $(1 - g_1) (1 - g_2) = 1 - g_1$ at the first step,
$\| 1 - g_2 \| \leq 1$
and~(\ref{4X02_Cut4}) at the third step,
(\ref{Eq_4X04_New}) and $\| 1 - g_1 \| \leq 1$
at the fifth step,
and the choice of $\dt$ at the sixth step,
\begin{align*}
& \big\| s - (1 - g_1) y_k (1 - g_1) \big\|
\\
& \alignInd
= \big\| (1 - g_2) \big( x_k^{(1)} \big)^{1/2} (1 - g_1) y_k (1 - g_1)
                        \big( x_k^{(1)} \big)^{1/2} (1 - g_2)
\\
& \alignInd  \alignInd
              - (1 - g_2) (1 - g_1) y_k (1 - g_1) (1 - g_2) \big\|
\\
& \alignInd
\leq \| 1 - g_2 \|^2
    \big\| \big( x_k^{(1)} \big)^{1/2} (1 - g_1) y_k (1 - g_1)
                  \big( x_k^{(1)} \big)^{1/2}
         - (1 - g_1) y_k (1 - g_1) \big\|
\\
& \alignInd
< 2 \dt + \big\| (1 - g_1) \big( x_k^{(1)} \big)^{1/2} y_k
                 \big( x_k^{(1)} \big)^{1/2} (1 - g_1)
         - (1 - g_1) y_k (1 - g_1) \big\|
\\
& \alignInd
\leq 2 \dt + \| 1 - g_1 \|^2
          \big\| \big( x_k^{(1)} \big)^{1/2} y_k
                 \big( x_k^{(1)} \big)^{1/2} - y_k \big\|
\\
& \alignInd
< 2 \dt + \frac{5}{24}
\leq \frac{1}{4}.
\end{align*}
So Corollary~1.6 of~\cite{Ph40}
implies that
\[
\big[ (1 - g_1) y_k (1 - g_1) - \tfrac{3}{4} \big]_{+}
  \subeq_A \big( s - \tfrac{1}{2} \big)_{+}.
\]
Clearly
\[
s
  \leq (1 - g_2) \big( x_k^{(1)} \big)^{1/2}
             \big( x_k^{(1)} \big)^{1/2} (1 - g_2)
  = r_k,
\]
so
\[
\big( s - \tfrac{1}{2} \big)_{+}
 \subeq_A \big( r_k - \tfrac{1}{2} \big)_{+}
\]
by Lemma~1.7 of~\cite{Ph40}.
The claim follows.

Using~(\ref{Eq_4X02_EqHalf}) and the definition $r_{k, N} = f_N (r_k)$,
we get
\[
r_{k, N} \big( r_k - \tfrac{1}{2} \big)_{+}
    = \big( r_k - \tfrac{1}{2} \big)_{+}.
\]
In particular,
$\big( r_k - \tfrac{1}{2} \big)_{+} \subeq_A r_{k, N}$.
Combining this relation with
(\ref{Eq_4X02_hSub}), (\ref{Eq_4X02_g0Sb}), and~(\ref{Eq_4X02_ykSbxk1}),
we get $1 - h \subeq_A r_{k, N}$.
This is part~(\ref{BothTimes2}) of the conclusion.

It remain to prove part~(\ref{BothTimes1}) of the conclusion.
It follows from~(\ref{4X02_Cut4})
that
\[
\big\| \big( x_k^{(1)} \big)^{1/2} (1 - g_2)^2
                   \big( x_k^{(1)} \big)^{1/2}
       - (1 - g_2) x_k^{(1)} (1 - g_2) \big\|
  < 2 \dt.
\]
Using $\dt \leq \dt_0$,
we get
\begin{equation}\label{Eq_4X02Last}
\big\| f_0 \big( \big( x_k^{(1)} \big)^{1/2} (1 - g_2)^2
                      \big( x_k^{(1)} \big)^{1/2} \big)
       - f_0 \big( (1 - g_2) x_k^{(1)} (1 - g_2) \big) \big\|
   < \frac{\ep}{4}.
\end{equation}
Since $x_k^{(0)} x_k^{(1)} = x_k^{(1)}$
(from~(\ref{Eq_4X02_xk0})),
we get
\[
x_k^{(0)}
   f_0 \big( \big( x_k^{(1)} \big)^{1/2} (1 - g_2)^2
                           \big( x_k^{(1)} \big)^{1/2} \big)
 = f_0 \big( \big( x_k^{(1)} \big)^{1/2} (1 - g_2)^2
                           \big( x_k^{(1)} \big)^{1/2} \big).
\]
Combining this equation with $\big\| x_k - x_k^{(0)} \big\| \leq 2 \ep_0$
(also from~(\ref{Eq_4X02_xk0}))
and using~(\ref{Eq_4X02Last}) twice,
we now get
\[
\big\| x_k
      f_0 \big( (1 - g_2) x_k^{(1)} (1 - g_2) \big) \big)
   - f_0 \big( (1 - g_2) x_k^{(1)} (1 - g_2) \big) \big\|
   < 2 \ep_0 + \frac{2 \ep}{4}
   \leq \ep.
\]
By the definition of $r_{k, 0}$,
we thus have $\| x_k r_{k, 0} - r_{k, 0} \| < \ep$,
as desired.
\end{proof}

\begin{thm}\label{tsr1GoesUpExtraLarge}
Suppose $A$ is an infinite dimensional simple separable
unital \ca.
Suppose $A$ has a \rokhLg{} subalgebra $B$
such that $\tsr (B) = 1$.
Then $\tsr (A) = 1$.
\end{thm}

\begin{proof}
We will show that every two sided zero divisor in
$A$ is a limit of invertible elements.
That is, if $a \in A$ and there are nonzero
$x, y \in A$ such that $x a = a y = 0$,
then we show that for every $\ep > 0$ there is an invertible element
$z \in A$ such that $\| z - a \| < \ep$.
It will follow from Theorem~3.3(a) of~\cite{Ro}
(see Definition~3.1 of~\cite{Ro})
that any $a \in A$ which is not a limit of
invertible elements is left or right invertible but not both.
Now $B$ has stable rank one,
so is finite by Proposition~3.1 of~\cite{Rf}.
Proposition 6.15 of~\cite{Ph40}
then implies that $A$ is finite,
so there are no such elements.
Thus we conclude that $A$ has stable rank one.

Without loss of generality,
$\| a \| \leq 1/2$ and $\ep \leq 1$.
Replacing $x$ with $\| x \|^{- 2} x^* x$
and $y$ with $\| y \|^{- 2} y y^*$,
we may assume that $x$ and $y$ are positive elements of norm~$1$.

Set $\delta_1 = \min (1, \ep) / 20$.
Apply Lemma~\ref{LemBothTimesAround} with $N = 1$,
with $n = 2$, $x_1 = x$, and $x_2 = y$,
with $m = 1$ and $a_1 = a$,
and with $\dt_1$ in place of~$\ep$,
to get
$a_0 \in A$,
$h_1 \in B$ with $0 \leq h_1 \leq 1$,
and $s_0, s_1, t_0, t_1 \in {\overline{h_1 B h_1}}$
such that:
\begin{enumerate}
\item\label{781}
$0 \leq s_0, s_1, t_0, t_1 \leq 1$
and $\| s_0 \| = \| s_1 \| = \| t_0 \| = \| t_1 \| = 1$.
\item\label{782}
$\| x s_0 - s_0 \| < \dt_1$ and $\| y t_0 - t_0 \| < \dt_1$.
\item\label{783}
$1 - h_1 \subeq_B s_1$ and $1 - h_1 \subeq_B t_1$.
\item\label{787}
$s_0 s_1 = s_1$ and $t_0 t_1 = t_1$.
\item\label{784}
$h_1 s_0 = s_0$ and $h_1 t_0 = t_0$.
\item\label{785}
$\| a - a_0 \| < \dt_1$.
\item\label{786}
$h_1 a_0 \in {\overline{h_1 B}} \subset B$.
\setcounter{TmpEnumi}{\value{enumi}}
\end{enumerate}

Set $a_1 = (1 - s_{0}) a_0 (1 - t_{0})$.
Then
$s_{1}a_1 = (s_{1} - s_{1} s_{0}) a_0 (1 - t_{0}) = 0$
by~(\ref{787}) and selfadjointness,
and similarly
$a_1 t_{1} = 0$.

Next we estimate $\| a_1 - a \|$.
Since $\| a \| \leq \frac{1}{2}$ and $x a = 0$,
we have
\[
\| s_{0} a_0 \|
\leq \| s_{0} \| \cdot \| a_0 - a \|
  + \| s_{0} - s_{0} x \| \cdot \| a \|
\leq 2 \delta_1.
\]
Similarly $\| a_0 t_{0}\| < 2 \delta_1$.
Now we can compute
\begin{align*}
\| a - a_1 \| &
\leq \| a - a_0 \| + \| a_0 - (1 - s_{0}) a_0 (1 - t_{0}) \|
\\
&
< \dt_1 + \| a_0 t_{0}\| + \| s_{0} a_0 \| \cdot \| 1 - t_{0} \|
\leq 5 \delta_1.
\end{align*}
This also gives us $\| a_1 \| < 5 \delta_1 + \frac{1}{2} \leq 1$.

The subalgebra $B$ is simple by Proposition~5.2 of~\cite{Ph40}.
Apply Lemma~\ref{Under2Ways},
getting $r \in B_{+} \setminus \{ 0 \}$ with
$r \leq s_{1}$ and $r \subeq_B t_{1}$.
Set $\dt_2 = \frac{1}{2} \| r \|$,
let $f_{\delta_2}$ be as in Definition~\ref{fSubEp},
and set $q = f_{\delta_2} (r)$.
Thus $\| q \| = 1$.
The functions $f_{\ep}$ of Definition~\ref{fSubEp}
are the same as those
defined at the beginning of Section~2 of~\cite{Ro2}.
Since $\tsr (B) = 1$,
Proposition 2.4(v) in~\cite{Ro2}
therefore provides a unitary $v \in B$ such that
$v^* q v \in {\ov{t_{1} B t_{1} }}$.

We claim that
\begin{equation}\label{Eq_6213_qa1vs}
q a_1 v^* = a_1 v^* q = 0.
\end{equation}
For the first,
use $q \in {\ov{s_{1} B s_{1} }}$ and $s_1 a_1 = 0$
to get $q a_1 = 0$.
For the second,
use $v^* q v \in {\ov{t_{1} B t_{1} }}$
and $a_1 t_1 = 0$ to get $a_1 v^* q v = 0$,
and multiply on the right by $v^*$
to get $a_1 v^* q = 0$.
The claim is proved.

Set
\[
\delta_3
 = \min \left( \frac{\ep}{20}, \,
   \frac{1}{4} \right).
\]
Apply Lemma~\ref{LemBothTimesAround} with $N = 1$,
with $n = 1$ and $x_1 = q$,
with $m = 1$ and $a_1 v^*$ in place of $a_1, a_2, \ldots, a_m$,
and with $\dt_3$ in place of~$\ep$,
to get
$a_2 \in A$,
$h_2 \in B$ with $0 \leq h_2 \leq 1$,
and $r_0, r_1 \in {\overline{h_2 B h_2}}$
such that:
\begin{enumerate}
\setcounter{enumi}{\value{TmpEnumi}}
\item\label{B781}
$0 \leq r_0, r_1 \leq 1$
and $\| r_0 \| = \| r_1 \| = 1$.
\item\label{B782}
$\| q r_0 - r_0 \| < \dt_3$.
\item\label{B783}
$1 - h_2 \subeq_B r_1$.
\item\label{B787}
$r_0 r_1 = r_1$.
\item\label{B784}
$h_2 r_0 = r_0$.
\item\label{B785}
$\| a_2 - a_1 v^* \| < \dt_3$.
\item\label{B786}
$h_2 a_2 \in {\overline{h_2 B}} \subset B$.
\end{enumerate}

Define $a_3 = (1 - r_{0}) a_2 (1 - r_{0})$.
Then $r_{1} a_3 = a_3 r_{1} = 0$.
Next,
\[
\| a_2 \|
 \leq \| a_2 - a_1 v^* \| + \| a_1 - a \| \cdot \| v^* \|
      + \| a \| \cdot \| v^* \|
 < \delta_3 + 5 \delta_1 + 1/2
 \leq 1.
\]
{}From~(\ref{B782}) and selfadjointness,
we get $\| r_0 - r_0 q \| < \dt_3$.
Using this and $q a_1 v^* = 0$
(from(\ref{Eq_6213_qa1vs})) at the second step,
we have
\begin{align*}
\| r_{0} a_2 \|
& \leq \| r_{0} a_2 - r_{0} a_1 v^* \|
    + \| r_{0} - r_{0} q \| \cdot \| a_1 v^* \| + \| r_{0} q a_1 v^* \|
\\
& \leq \delta_3 \| r_{0} \| + \delta_3 \| a_1 \| \cdot \| v^* \|
  \leq 2 \dt_3.
\end{align*}
Similarly,
$a_1 v^* q = 0$ (also from(\ref{Eq_6213_qa1vs})) gives
$\| a_2 r_{0} \| \leq 2 \dt_3$.

Now we can estimate:
\[
\| a_2 - a_3 \|
  = \| a_2 - (1 - r_{0}) a_2 (1 - r_{0}) \|
  \leq \| a_2 r_{0} \| + \| r_{0} a_2 \| \cdot \| 1 - r_{0} \|
  < 4 \delta_3.
\]
Therefore
\[
\| a_3 v - a \|
 \leq \| a_3 - a_2 \| \cdot \| v \| + \| a_2 v -  a_1 \|
   + \| a_1 - a  \|
 < 4 \delta_3 + \delta_3 + 5 \delta_1
 \leq \frac{\ep}{2}.
\]

Combining~(\ref{B784}) with its adjoint,
we get $h_2 r_0 = r_0 h_2$.
So $h_2^{1 / 2} r_0 = r_0 h_2^{1 / 2}$.
Therefore
\[
h_2 a_3
 =  h_2^{1/2}
   ( 1 - r_{0} ) h_2^{1/2} a_2 (1 - r_{0}).
\]
We have $h_2^{1/2} a_2 \in B$ by~(\ref{B786})
and Corollary~\ref{C_4926_InSub},
and $1 - r_{0} \in B$,
so
\begin{equation}\label{Eq_4926_h2B}
h_2 a_3
 \in {\ov{h_2 B}} \cdot {\ov{h_2 B}} \cdot B
 \subset {\ov{h_2 B}}
 \subset B.
\end{equation}

If $h_2 = 1$,
then $a_3 \in B$.
Since $\tsr (B) = 1$,
there is an invertible element $z_0 \in B$
such that $\| z_0 - a_3 \| < \ep - \| a_3 v - a \|$.
Set $z = z_0 v$.
Then $z$ is invertible in~$A$
and $\| z - a \| < \ep$,
as desired.
So we may assume that $h_2 \neq 1$.
Choose $\dt_4 > 0$ such that
$\spec (h_2)$ is not contained in $[1 - 2 \dt_4, \, 1]$.
In particular, $\dt_4 < \frac{1}{2}$.
Let $f_{\delta_4}$ be as in Definition~\ref{fSubEp}.
Since $1 - h_2 \geq 0$ and $\spec (1 - h_2)$
is not contained in $[0, \, 2 \dt_4]$,
we have $\| f_{\delta_4} (1 - h_2) \| = 1$.

Using~(\ref{B783}),
$\tsr (B) = 1$,
and Proposition 2.4(v) in~\cite{Ro2},
we get a unitary
$u \in U (B)$ such that
\begin{equation}\label{Eq_4925NStar}
u f_{\delta_4} (1 - h_2) u^* \in {\ov{r_{1} B r_{1} }}.
\end{equation}
Define
\[
z_1 = u f_{\delta_4} (1 - h_2 ) u^*,
\,\,\,\,\,\,
z_3 = f_{\delta_4} (1 - h_2 ),
\andeqn
z_2 = 1 - z_1 - z_2.
\]
We want to apply Lemma~\ref{bStr1}
with $A$, $B$, $z_{1}$, $z_{2}$, and $z_{3}$
as given,
and with $a_3 u$ in place of~$a$.

We have $\tsr (B) = 1$ by hypothesis,
and hypothesis~(\ref{26641}) of Lemma~\ref{bStr1}
holds by definition.
Also $z_{1}, z_{2}, z_{3} \in B$ by construction.

Next,
we claim that $z_1z_3 = 0$.
{}From
(\ref{B787}) and~(\ref{B784}),
we get $h_2 r_1 = r_1$,
so $(1 - h_2) r_1 = 0$.
Thus
$z_1 (1 - h_2) = u f_{\delta_4} (1 - h_2) u^* (1 - h_2) = 0$
by Lemma~\ref{OrthHer},
and another application of Lemma~\ref{OrthHer} gives
\begin{equation}\label{Eq_4926_zero}
z_1 z_3
 = z_1 f_{\delta_4} (1 - h_2) z_1
 = 0,
\end{equation}
proving the claim.
This is hypothesis~(\ref{26642}) of Lemma~\ref{bStr1}.
Clearly $0 \leq z_1 \leq 1$
and $0 \leq z_3 \leq 1$,
so we get $z_2 \geq 0$.
Thus hypothesis~(\ref{26640}) of Lemma~\ref{bStr1} holds.

Combining (\ref{Eq_4925NStar}) with~(\ref{B787})
and the definition of~$z_1$,
we get $r_0 z_1 = z_1 r_0 = z_1$.
So
\[
z_1 a_3 u = z_1 (1 - r_0) a_2 (1 - r_0) u = 0,
\]
which is hypothesis~(\ref{26643}) of Lemma~\ref{bStr1}.
Clearly $u z_3 = z_1 u$,
so also
\[
a_3 u z_3
 = a_3 z_1u
 = (1 - r_0) a_2 (1 - r_0) z_1 u
 = 0,
\]
which is hypothesis~(\ref{26644}) of Lemma~\ref{bStr1}.

It remains only to verify
hypothesis~(\ref{26645}) of Lemma~\ref{bStr1}.
Recall from~(\ref{Eq_4926_h2B}) that $h_2 a_3 \in B$.
Since $\dt_4 < \frac{1}{2}$,
the function
$\ld \mapsto 1 - f_{\delta_4} (1 - \ld)$,
from $[0, 1]$ to $[0, 1]$,
vanishes at zero.
So Corollary~\ref{C_4926_InSub} implies
$(1 - z_3) a_3 = [ 1 - f_{\delta_4} (1 - h_2) ] a_3 \in B$.
{}From~(\ref{Eq_4926_zero}) and $z_1, u, z_2 \in B$,
we now get
$z_2 a_3 u z_2 = (1- z_1) (1 - z_3) a_3 u z_2 \in B$,
as desired.

We can now apply Lemma~\ref{bStr1}
to find an invertible element $a_4 \in A$
such that $\| a_4 - a_3 u \| < \tfrac{\ep}{2}$.
Then $z = a_4 u^* v$ is invertible and satisfies
\[
\| z - a \|
  \leq \| a_4 - a_3 u \| \cdot \| u^* \| \cdot \| v \|
         + \| a_3 v - a \|
  < \tfrac{\ep}{2} + \tfrac{\ep}{2}
  = \ep.
\]
This completes the proof.
\end{proof} 

Combining Theorem~\ref{tsr1GoesUpExtraLarge}
with results of~\cite{CEI} and~\cite{Ph40},
we can show that if, in addition, the centrally large subalgebra
has real rank zero,
then so does the containing algebra.
It is possible that one does not need the subalgebra
to have stable rank one,
but the proof would need to be much longer.

\begin{thm}\label{T_4Y11_RRZ}
Suppose $A$ is an infinite dimensional simple separable
unital \ca.
Suppose $A$ has a \rokhLg{} subalgebra~$B$
such that $\tsr (B) = 1$ and $\RR (B) = 0$.
Then $\RR (A) = 0$.
\end{thm}

\begin{proof}
We use the Cuntz semigroups $\Cu (B)$ and $\Cu (A)$.
We refer to \cite{APT} and Section~1 of~\cite{Ph40} for background.
In particular, recall the compact containment relation $\ll$,
Definition 1.24 and Lemma 1.25 of~\cite{Ph40}
or (under the name ``way below'')
condition~{\textbf{(O4)}} in Definition~4.1 of~\cite{APT}.
If $D$ is a \ca,
we further call an element $\et \in \Cu (D)$
compact if $\et \ll \et$.

Since $\RR (B) = 0$ and $\tsr (B) = 1$,
it follows from the second paragraph of Corollary~5
on page 186 of~\cite{CEI}
that every element of $\Cu (B)$ is the supremum
of a nondecreasing sequence of compact elements.
We claim that the same is true of~$A$.
So let $\et \in \Cu (A)$.
If $\et$ is the class of a \pj{} in $K \otimes A$,
then $\et$ is itself compact,
by Lemma 1.25(3) of~\cite{Ph40}.
So the claim is immediate.
Otherwise,
let $\io \colon B \to A$ be the inclusion map.
It follows from Theorem~\ref{tsr1GoesUpExtraLarge}
that $\tsr (A) = 1$,
so $A$ is stably finite
by Proposition~3.1 and Theorem~3.3 of~\cite{Rf}.
Since centrally large algebras are large,
Corollary~\ref{C_4920_StFinStZLg}
now implies that $B$ is stably large in~$A$.
So Theorem~6.6 of~\cite{Ph40}
provides $\mu \in \Cu (B)$ such that $\io_* (\mu) = \et$.
Let $(\mu_n)_{n \in \N}$
be a nondecreasing sequence of compact elements in $\Cu (B)$
such that $\mu = \sup_{n \in \N} \mu_n$.
Then $(\io_* (\mu_n) )_{n \in \N}$
is a nondecreasing sequence of compact elements in $\Cu (A)$
such that $\et = \sup_{n \in \N} \io_* (\mu_n)$.
This proves the claim.

We saw already that $\tsr (A) = 1$.
Combining this fact, the claim,
and Corollary~5 on page 186 of~\cite{CEI},
we get $\RR (A) = 0$.
\end{proof}


\section{Application to the stable rank of crossed
 products}\label{Sec_TsrCp}

In this section,
we apply our main result to show that the crossed product
by a minimal homeomorphism of a compact metric space~$X$
has stable rank one whenever
there is a continuous surjective map from $X$ to the Cantor set.
We conjecture that the result holds whenever $X$ is infinite.
Our theorem covers examples in which it is known that
the crossed product is not $Z$-stable,
and in particular is not classifiable
in terms of K-theory and traces.
In particular,
our result can't be proved by using the fact
that simple unital $Z$-stable \ca{s} have stable rank one
(Theorem~6.7 of~\cite{R4}).

\begin{thm}\label{T_5413_TsrCP}
Let $X$ be a compact metric space.
Assume that
there is a continuous surjective map from $X$ to the Cantor set.
Let $h \colon X \to X$ be a minimal homeomorphism.
Then $C^* (\Z, X, h)$ has stable rank one.
\end{thm}

In Proposition~5.1 of~\cite{HPT},
it is shown that the conditions on the pair $(X, h)$
are equivalent to a number of other conditions.
For example,
the hypotheses are equivalent to the assumption
that $(X, h)$ has a factor system
which is a \hme{} of the Cantor set.
They are also equivalent to the assumption
that there is a decreasing sequence
$Y_0 \supset Y_1 \supset \cdots$
of nonempty compact open subsets of~$X$ such that the subset
$Y = \bigcap_{n = 0}^{\infty} Y_n$
satisfies
$h^r (Y) \cap Y = \varnothing$ for all $r \in \Z \setminus \{ 0 \}$.

The proof depends on a result in~\cite{EHT},
according to which simple direct limits of AH~systems with diagonal maps
have stable rank one,
regardless of any conditions on dimension growth.
It seems plausible to believe that the same should be true
for simple direct limits of recursive subhomogeneous algebras
in which the maps of the system are diagonal.
For any \mh~$h$ of an infinite compact metric space,
and any nonempty closed subset $Y \subset X$
such that
$h^r (Y) \cap Y = \varnothing$ for all $r \in \Z \setminus \{ 0 \}$,
the centrally large subalgebra $C^* (\Z, X, h)_Y$
(see the introduction for the notation)
is the direct limit of such a system.
(This will appear in~\cite{Ph_md}.)
Accordingly,
we make the following conjecture.

\begin{cnj}[joint with Zhuang Niu]\label{Cj_4Z05_TsrMH}
Let $X$ be an infinite \cms{}
and let $h \colon X \to X$ be a \mh.
Then $C^* (\Z, X, h)$ {\emph{always}} has stable rank one.
\end{cnj}

Since simple unital AH~algebras can have arbitrary
stable rank
(even infinite stable rank;
see Theorems 8 and~12 of~\cite{Vl}),
this would show that some kinds of pathology
which occur for simple unital AH~algebras
are not possible with crossed products by \mh{s}.

\begin{proof}[Proof of Theorem~\ref{T_5413_TsrCP}]
Proposition 5.14 of~\cite{HPT}
provides a nonempty compact subset $Y \subset X$
such that
$h^n (Y) \cap Y = \varnothing$ for all $n \in \Z \setminus \{ 0 \}$
and such that $C^* (\Z, X, h)_Y$
(as defined in the introduction)
is the direct limit of an AH~system
with diagonal maps in the sense described at the
beginning of Section~2.2 of~\cite{EHT}.
(The maps are explicitly described in Lemma 5.12 of~\cite{HPT}.)
By Theorem~7.10 of~\cite{Ph40},
the subalgebra
$C^* (\Z, X, h)_Y$ is a large subalgebra of $C^* (\Z, X, h)$
of crossed product type in the sense of
Definition~4.9 of~\cite{Ph40}.
In particular,
$C^* (\Z, X, h)_Y$ is simple
(by Proposition~5.2 of~\cite{Ph40})
and is centrally large in~$C^* (\Z, X, h)$
(by Theorem~\ref{T-2717CptImpRq}).
Using simplicity and the diagonal maps condition
in the AH~system,
it follows from Theorem~4.1 of~\cite{EHT}
that $C^* (\Z, X, h)_Y$ has stable rank one.
So $C^* (\Z, X, h)$ has stable rank one
by Theorem~\ref{tsr1GoesUpExtraLarge}.
\end{proof}

\begin{cor}\label{C_5413_GK}
Let $Y$ be a compact metric space with more than one point,
let $d \in (0, 1)$,
and let $X$ and the \mh{} $h \colon X \to X$ be a
subshift
(of the shift on $(Y \times [0, 1])^{\Z}$)
as constructed in Section~2 of~\cite{GK}.
Then $C^* (\Z, X, h)$ has stable rank one.
\end{cor}

\begin{proof}
On the bottom of page 111 of~\cite{GK},
it is proved that the system $(X, h)$
has the universal odometer as a factor system.
(The hypothesis that $Y$ be contractible,
which is being assumed at this point,
is not needed for this conclusion.)
In particular,
there is a continuous surjective map from $X$ to the Cantor set.
So Theorem~\ref{T_5413_TsrCP} applies.
\end{proof}

\begin{exa}\label{Ex_5413_GKrc}
Let $n \in \N$ satisfy $n \geq 2$.
In Corollary~\ref{C_5413_GK},
make the particular choices
$d \in \big( 1 - \frac{1}{n}, \, 1 \big)$
and $Y = [0, 1]^{3 n}$.
Theorem~2.2 of~\cite{GK}
implies that
the radius of comparison of $C^* (\Z, X, h)$
satisfies
${\operatorname{rc}} ( C^* (\Z, X, h)) \geq n - 1$.
In particular,
$C^* (\Z, X, h)$ does not have strict comparison,
so is not $Z$-stable
(by Corollary~4.6 of~\cite{R4}),
but has stable rank one by Corollary~\ref{C_5413_GK}.
\end{exa}

The {\emph{mean dimension}} ${\operatorname{mdim}} (h)$
of a \hme{} $h$ of a \cms{}
is defined in Definition~2.6 of~\cite{LW}.
It is shown in~\cite{HPT}
that for the \mh{s} in
Corollary~\ref{C_5413_GK}
one actually has
${\operatorname{rc}} ( C^* (\Z, X, h))
 = \frac{1}{2} {\operatorname{mdim}} (h)$.
It is pointed out after
Proposition~2.4 of~\cite{GK}
that if the covering dimension
of the space $Y$ in Corollary~\ref{C_5413_GK}
satisfies $\dim (Y) > 0$,
then the \hme{} $h$ of Corollary~\ref{C_5413_GK}
satisfies ${\operatorname{mdim}} (h) > 0$.
Thus,
in fact,
whenever $\dim (Y) > 0$ in Corollary~\ref{C_5413_GK},
then ${\operatorname{rc}} ( C^* (\Z, X, h)) > 0$,
and $C^* (\Z, X, h)$ is not $Z$-stable
(but does have stable rank one).

\begin{exa}\label{Ex_5413_GKPerf}
In Corollary~\ref{C_5413_GK},
make the particular choices
$d \in (0, 1)$
and $Y = S^2 \times S^2$.
Theorem~3.1 of~\cite{GK}
implies that $K_0 (C^* (\Z, X, h))$
is not weakly unperforated.
We can now give a different argument to show that
$C^* (\Z, X, h)$ is not $Z$-stable.
By Corollary~4.9 of~\cite{R4},
if $C^* (\Z, X, h)$ were $Z$-stable
then (using simplicity of $C^* (\Z, X, h)$
to get simplicity of $V (C^* (\Z, X, h))$)
it would follow that $K_0 (C^* (\Z, X, h))$
is almost unperforated.
By the discussion after Lemma~3.4 of~\cite{R4},
it would follow that $K_0 (C^* (\Z, X, h))$ is weakly unperforated,
a contradiction.
Again,
$C^* (\Z, X, h)$ has stable rank one by Corollary~\ref{C_5413_GK}.
\end{exa}

There are other ways to construct examples of \mh{s}
for which we can prove that the mean dimension is nonzero
and that the crossed products have stable rank one,
although we do not know that the crossed products
don't have strict comparison,
and we do not know that they are not $Z$-stable.

\begin{exa}\label{Ex_5413_Prod}
Let $X_0$ be any connected compact metric space,
and let $h_0 \colon X_0 \to X_0$ be a minimal homeomorphism
such that ${\operatorname{mdim}} (h_0) > 0$.
Let $k \colon Z \to Z$ be an odometer homeomorphism
of the Cantor set.
(See, for example, Section VIII.4 of~\cite{Dv}.)
Set $X = Z \times X_0$ and define $h = k \times h_0 \colon X \to X$.
It is shown in~\cite{HPT}
that $h$ is minimal.
It is also shown in~\cite{HPT}
that ${\operatorname{mdim}} (h) = {\operatorname{mdim}} (h_0)$.
Obviously $X$ has a continuous surjective map to the Cantor set,
so $C^* (\Z, X, h)$ has stable rank one
by Theorem~\ref{tsr1GoesUpExtraLarge}.
\end{exa}


\renewcommand{\bibname}{\textsc{REFERENCES}}

\begin{thebibliography}{10}

\bibitem{ArPh2} D.~Archey and N.~C.\  Phillips,
{\emph{Centrally large subalgebras and $Z$-stability}},
in preparation.

\bibitem{APT} P.~Ara, F.~Perera, and A.~S.\  Toms,
{\emph{K-Theory for operator algebras. Classification of C*-algebras}},
pages 1--71 in:
{\emph{Aspects of Operator Algebras and Applications}},
P.~Ara, F~Lled\'{o}, and F.~Perera (eds.),
Contemporary Mathematics vol.~534,
Amer.\  Math.\  Soc., Providence RI, 2011.

\bibitem{CEI} K.~T.\  Coward, G.~A.\  Elliott, and C.~Ivanescu,
{\emph{The Cuntz semigroup as an invariant for C*-algebras}},
J.~reine angew.\  Math.\  {\textbf{623}}(2008), 161--193.

\bibitem{Cz}
J.~Cuntz,
{\emph{The structure of multiplication and addition in
simple C*-algebras}},
Math.\  Scand.\  {\textbf{40}}(1977),
215--233.

\bibitem{Cz1}
J.~Cuntz,
{\emph{Dimension functions on simple C*-algebras}},
Math.\  Ann.\  {\textbf{233}}(1978),
145--153.

\bibitem{Dv} K.~R.\  Davidson,
{\emph{C*-Algebras by Example}},
Fields Institute Monographs no.~6,
Amer.\  Math.\  Soc., Providence RI, 1996.

\bibitem{EHT}
G.~A.\  Elliott, T.~M.\  Ho, and A.~S.\  Toms,
{\emph{A class of simple C*-algebras with stable rank one}},
J.~Funct.\  Anal.\  {\textbf{256}}(2009), 307--322.

\bibitem{EN1} G.~A.\  Elliott and Z.~Niu,
{\emph{All irrational extended rotation algebras are AF~algebras}},
preprint.

\bibitem{EN2} G.~A.\  Elliott and Z.~Niu,
{\emph{C*-algebra of a minimal homeomorphism of zero mean dimension}},
preprint (arXiv:1406.2382v2 [math.OA]).

\bibitem{GK} J.~Giol and D.~Kerr,
{\emph{Subshifts and perforation}},
J.~reine angew.\  Math.\  {\textbf{639}}(2010), 107--119.

\bibitem{HPT} T.~Hines, N.~C.\  Phillips, and A.~S.\  Toms,
{\emph{Mean dimension and radius of comparison for minimal
 homeomorphisms with Cantor factors}},
in preparation.

\bibitem{KR}
E.~Kirchberg and M.~R{\o}rdam,
{\emph{Non-simple purely infinite C*-algebras}},
Amer.\  J.\  Math.\  {\textbf{122}}(2000),
637--666.

\bibitem{LnBook} H.~Lin,
{\emph{An Introduction to the Classification of Amenable
C*-Algebras}},
World Scientific, River Edge~NJ, 2001.

\bibitem{LW} E.~Lindenstrauss and B.~Weiss,
{\emph{Mean topological dimension}},
Israel J.\  Math.\  {\textbf{115}}(2000), 1--24.

\bibitem{Ph10}  N.~C.\  Phillips,
{\emph{Crossed products of the Cantor set by free minimal
actions of ${\mathbb{Z}}^d$}},
Commun.\  Math.\  Phys.\  {\textbf{256}}(2005),
1--42.

\bibitem{Ph40}
N.~C.\  Phillips,
{\emph{Large subalgebras}},
preprint (arXiv: 1408.5546v1 [math.OA]).

\bibitem{Ph41}
N.~C.\  Phillips,
{\emph{The transformation group C*-algebras of free minimal actions of
${\mathbb{Z}}^d$ on finite dimensional compact metric spaces (draft)}}.

\bibitem{PhZd}
N.~C.\  Phillips,
{\emph{Stable rank one for free minimal actions of ${\mathbb{Z}}^d$
  with Cantor factors}},
in preparation.

\bibitem{Ph_md}
N.~C.\  Phillips,
in preparation.

\bibitem{Pt1} I.~F.\  Putnam,
{\emph{The C*-algebras associated with minimal
homeomorphisms of the Cantor set}},
Pacific J.\  Math.\  {\textbf{136}}(1989), 329--353.

\bibitem{Pt4} I.~F.\  Putnam,
{\emph{On the K-theory of C*-algebras of principal groupoids}},
Rocky Mountain J.\  Math.\  {\textbf{28}}(1998), 1483--1518.

\bibitem{Rf} M.~A.\  Rieffel,
{\emph{Dimension and stable rank in the K-theory of C*-algebras}},
Proc.\  London Math.\  Soc.~(3) {\textbf{46}}(1983), 301--333.

\bibitem{Ro}
M.~R{\o}rdam,
{\emph{On the structure of simple C*-algebras tensored with
   a UHF-algebra}},
J.~Funct.\  Anal.\  {\textbf{100}}(1991),
1--17.

\bibitem{Ro2}
M.~R{\o}rdam,
{\emph{On the structure of simple C*-algebras tensored with
   a UHF-algebra. II}},
J.~Funct.\  Anal.\  {\textbf{107}}(1992),
255--269.

\bibitem{R4} M.~R{\o}rdam,
{\emph{The stable and the real rank of ${\mathcal{Z}}$-absorbing
C*-algebras}},
International J.\  Math.\  {\textbf{15}}(2004), 1065--1084.

\bibitem{RW} M.~R{\o}rdam and W.~Winter,
{\emph{The Jiang-Su algebra revisited}},
J.~reine angew.\  Math.\  {\textbf{642}}(2010), 129--155.

\bibitem{Szb} G.~Szab\'{o},
{\emph{The Rokhlin dimension of topological ${\mathbb{Z}}^m$-actions}},
preprint (arXiv: 1308.5418v4 [math.OA]).

\bibitem{Vl} J.~Villadsen,
{\emph{On the stable rank of simple C*-algebras}},
J.~Amer.\  Math.\  Soc.\  {\textbf{12}}(1999),  1091--1102.

\end{thebibliography}


\end{document}